\documentclass[11pt]{amsart}
\usepackage{amssymb, latexsym, mathrsfs,   color }
\usepackage[all]{xy}
\usepackage[colorlinks=true, pdfstartview=FitV,linkcolor=blue,citecolor=blue,urlcolor=blue]{hyperref}

    \advance \topmargin by -\headheight
    \advance \topmargin by -\headsep

    \evensidemargin \oddsidemargin
\newtheorem {theorem}    {Theorem}[section]

\newtheorem {lemma}      [theorem]    {Lemma}
\newtheorem {corollary}  [theorem]    {Corollary}
\newtheorem {proposition}[theorem]    {Proposition}

\newcommand{\bb}{\mathbb}
\renewcommand{\rm}{\mathrm}

\newcommand{\GG}{\mathrm{G}}
\newcommand{\GGL}{\mathrm{GL}}
\newcommand{\UU}{\mathrm{U}}
\newcommand{\Fq}{\bb{F}_q}
\newcommand{\fq}{(\bb{F}_q)}
\theoremstyle{definition}

\newcommand{\CD}{{\mathcal{D}}}
\newcommand{\CQ}{{\mathcal{Q}}}
\newcommand{\CJ}{{\mathcal {J}}}
\newcommand{\UUl}{\underline}

\numberwithin{equation}{section}

\begin{document}

\title{Descents of unipotent representations of finite unitary groups}

\date{\today}

\author[Dongwen Liu]{Dongwen Liu}

\address{School of Mathematical Science, Zhejiang University, Hangzhou 310027, Zhejiang, P.R. China}

\email{maliu@zju.edu.cn}

\author[Zhicheng Wang]{Zhicheng Wang}

\address{School of Mathematical Science, Zhejiang University, Hangzhou 310027, Zhejiang, P.R. China}

\email{11735009@zju.edu.cn}
\subjclass[2010]{Primary 20C33; Secondary 22E50}

\begin{abstract}
Inspired by the Gan-Gross-Prasad conjecture and the descent problem for classical groups, in this paper we study the descents of unipotent representations of unitary groups over finite fields. We give the first descents of unipotent representations explicitly, which are unipotent as well. Our results include both the Bessel case and Fourier-Jacobi case, which are related via theta correspondence.
\end{abstract}

\maketitle

\section{Introduction}

In representation theory, a classical problem is to look for the spectral decomposition
of a representation $\pi$ of a group $G$ restricted to a subgroup $H$. Namely, one asks for which representation $\sigma$ of $H$ has the property that
\[
\rm{Hom}_H(\pi,\sigma)\neq 0,
\]
and what the dimension of this Hom-space is. In general such a restriction problem is hard and may not have reasonable answers. However when $G$ is
a classical group defined over a local field and $\pi$ belongs to a generic Vogan $L$-packet, the local Gan-Gross-Prasad conjecture \cite{GP1, GP2, GGP1} provides explicit answers and is one of the most successful examples concerning with those general questions. To be a little more precise, the multiplicity one property holds for this situation, namely
\[
m(\pi,\sigma):=\dim\rm{Hom}_H(\pi,\sigma)\leq 1,
\]
and the invariants attached to $\pi$ and $\sigma$ that detect the multiplicity $m(\pi,\sigma)$ is the local root number associated to their Langlands parameters.  In the $p$-adic case, the local Gan-Gross-Prasad conjecture  has been resolved by J.-L. Waldspurger and C. M\oe glin and J.-L. Waldspurger \cite{W1, W2, W3, MW} for orthogonal groups, by R. Beuzart-Plessis \cite{BP1, BP2} and W. T. Gan and A. Ichino \cite{GI} for unitary groups, and by H. Atobe \cite{Ato} for symplectic-metaplectic groups. On the other hand, D. Jiang and L. Zhang \cite{JZ1} study the local descents for $p$-adic orthogonal groups, whose results can be viewed as a refinement of the local Gan-Gross-Prasad conjecture, and the descent method has important applications towards the global problem (see \cite{JZ2}).

Motivated by the above works, in this paper which is the first one of a series, we will study the descent problem for unipotent representations of finite unitary groups. To begin with, we first set up some notations. Let $\overline{\mathbb{F}}_q$ be an algebraic closure of a finite field $\mathbb{F}_q$, which is of characteristic $p>2$. Let $G=\UU_n$ be an $\bb{F}_q$-rational form of $\GGL_n(\overline{\mathbb{F}}_q)$, and $F$ be the corresponding Frobenius endomorphism, so that the group of $\Fq$-rational points  is
$G^F=\UU_n(\Fq)$. Let $Z$ be the center of $G^F$. We  assume that $q$ is large enough such that the main theorem in \cite{S2} holds, namely assume that
 \begin{itemize}
 \item $q$ is large enough such that $T^F/Z$ has at least two Weyl group orbits of regular characters, for every $F$-stable maximal torus $T$ of $G$.
 \end{itemize}
 However, this assumption will be only necessary for the Fourier-Jacobi case of our main result, which will be proven using theta correspondence.
For an $F$-stable maximal torus $T$ of $G$ and a character $\theta$ of $T^F$,  let $R_{T,\theta}^G$ be the virtual character of $G^F$ defined by P. Deligne and G. Lusztig in their seminal work \cite{DL}. An irreducible representation $\pi$ is called unipotent if there is an $F$-stable maximal torus $T$ of $G$ such that $\pi$ appears in $R_{T,1}^G$. For two representations $\pi$ and $\pi'$ of a  finite group $H$, define
\[
\langle \pi,\pi'\rangle_H:=\dim \mathrm{Hom}_H(\pi,\pi').
\]

Let $\pi$ and $\pi'$ be  irreducible representations of $\UU_{n}(\Fq)$ and $\UU_{m}(\Fq)$ respectively, where $n\ge m$.
The Gan-Gross-Prasad conjecture is concerned with the multiplicity
\[
m(\pi, \pi'):=\langle \pi\otimes\bar{\nu},\pi' \rangle_{H(\Fq)} = \dim \mathrm{Hom}_{H(\Fq)}(\pi\otimes\bar{\nu},\pi' )
\]
where the data $(H, \nu)$ is defined as in \cite[Theorem 15.1]{GGP1}, and will be explained in details shortly. According to whether $n-m$ is odd or even, the above-Hom space is called the Bessel model or Fourier-Jacobi model. In \cite[Proposition 5.3]{GGP2}, W. T. Gan, B. H. Gross and D. Prasad showed that if both $\pi$ and $\pi'$ are cuspidal, then
\[
m(\pi,\pi') \le 1.
\]
We should mention that our formulation of multiplicities differs slightly from that in the Gan-Gross-Prasad conjecture \cite{GGP1}, up to taking the contragradient of $\pi'$. This is more suitable for the purpose of descents, which will be clear from the discussion below. On the other hand, in this paper we will restrict our attention to unipotent representations of $\UU_{n}(\Fq)$, which are self-dual (see Proposition \ref{7.4}) and thus for $\pi$ unipotent the above Hom-space coincides with $\rm{Hom}_{H(\Fq)}(\pi\otimes\pi',\nu)$.

Roughly speaking,  for fixed $\UU_n(\Fq)$ and its representation $\pi$,  the descent problem seeks the smallest member $\UU_m(\Fq)$ among a Witt tower which has an irreducible representation $\pi'$ satisfying $m(\pi, \pi')\neq 0$, and all such $\pi'$ give the first descent of $\pi$. To give the precise notion of descent, let us sketch the definition of the data $(H,\nu)$ following \cite{GGP1} and \cite{JZ1}.

We first consider the Bessel model. Let $V_n$ be an $n$-dimensional space over $\mathbb{F}_{q^2}$ with a nondegenerate Hermitian form $(,)$, which defines the unitary group $\UU_n(\Fq)$. Consider pairs of Hermitian spaces $V_n\supset V_{n-2\ell-1}$ and the following partitions of $n$,
\begin{equation}\label{pl}
\UUl{p}_\ell=[2\ell+1, 1^{n-2\ell-1}],\quad 0\leq \ell<n/2.
\end{equation}
Assume that  $V_n$ has a decomposition
\[
V_n=X+V_{n-2\ell}+X^\vee
\]
where $X+X^\vee=V_{n-2\ell}^\perp$ is a polarization. Let $\{e_1,\ldots, e_\ell\}$ be a basis of $X$, $\{e_1',\ldots, e_\ell'\}$ be the dual basis of $X^\vee$, and
let $X_i=\rm{Span}_{\mathbb{F}_{q^2}}\{e_1,\ldots, e_i\}$, $i=1,\ldots, \ell$. Let $P$ be the parabolic subgroup of $\UU_n$ stabilizing the flag
\[
X_1\subset\cdots\subset X_\ell,
\]
so that its Levi component is $M\cong \rm{Res}_{\mathbb{F}_{q^2}/\Fq}\GGL_1^\ell\times \rm{U}_{n-2\ell}$. Its unipotent radical can be written in the form
\[
N_{\UUl{p}_\ell}=\left\{n=
\begin{pmatrix}
z & y & x\\
0 & I_{n-2\ell} & y'\\
0 & 0 & z^*
\end{pmatrix}
: z\in U_{\GG_\ell}
\right\},
\]
where the superscript ${}^*$ means the conjugate transpose, and $U_{\GG_\ell}$ is the subgroup of unipotent upper triangular matrices of $\GG_\ell:=\rm{Res}_{\mathbb{F}_{q^2}/\Fq}\GGL_\ell$.
Fix a nontrivial additive character $\psi_{\mathbb{F}_{q^2}}$ of $\mathbb{F}_{q^2}$. Pick up an anisotropic vector $v_0\in V_{n-2\ell}$ and define a generic character $\psi_{\UUl{p}_\ell, v_0}$ of $N_{\UUl{p}_\ell}(\Fq)$ by
\[
\psi_{\UUl{p}_\ell, v_0}(n)=\psi_{\mathbb{F}_{q^2}}\left(\sum^{\ell-1}_{i=1}z_{i,i+1}+(y_\ell, v_0)\right), \quad n\in N_{\UUl{p}_\ell}(\Fq),
\]
where $y_\ell$ is the last row of $y$. The stabilizer of $\psi_{\UUl{p}_\ell, v_0}$ in $M(\Fq)$ is  the unitary group of the orthogonal complement of $v_0$ in $V_{n-2\ell}$, which will be identified with $\UU_{n-2\ell-1}(\Fq)$.
Put
\begin{equation}\label{hnu}
H=\rm{U}_{n-2\ell-1}\ltimes N_{\UUl{p}_\ell},\quad \nu=\psi_{\UUl{p}_\ell, v_0}.
\end{equation}
For irreducible representations $\pi$ and $\pi'$ of $\rm{U}_n(\Fq)$ and $\rm{U}_{n-2\ell-1}(\Fq)$ respectively, the uniqueness of
{\sl Bessel models} asserts that
\[
m(\pi,\pi')=\dim \rm{Hom}_{H(\Fq)}(\pi\otimes\bar{\nu}, \pi')\leq 1.
\]
 This was proved over $p$-adic local fields in \cite{AGRS}, and for cupsidal representations over finite fields in \cite{GGP2}. It is clear that
\[
\rm{Hom}_{H(\Fq)}(\pi\otimes\bar{\nu},\pi')
\cong
\rm{Hom}_{\rm{U}_{n-2\ell-1}(\Fq)}(\CJ_{\ell}(\pi),\pi'),
\]
where $\CJ_\ell(\pi)$ is the twisted Jacquet module of $\pi$ with respect to $(N_{\UUl{p}_\ell}(\Fq), \psi_{\UUl{p}_{\ell}, v_0})$.
We simply define the notion of the $\ell$-th {\sl Bessel quotient} of $\pi$ by
\begin{equation}\label{lbd}
\CQ^\rm{B}_{\ell}(\pi):=\CJ_{\ell}(\pi),
\end{equation}
viewed as a representation of $\UU_{n-2\ell-1}(\Fq)$. Define the {\sl first occurrence index} $\ell_0:=\ell_0^\rm{B}(\pi)$ of $\pi$ in the Bessel case to be the largest nonnegative integer $\ell_0<n/2$ such that $\CQ^\rm{B}_{\ell_0}(\pi)\neq 0$. The $\ell_0$-th Bessel descent of $\pi$ is called the {\sl first Bessel descent} of $\pi$ or simply the {\sl Bessel descent} of $\pi$, denoted
\begin{equation}\label{bd}
\CD^\rm{B}_{\ell_0}(\pi):=\CQ^\rm{B}_{\ell_0}(\pi).
\end{equation}

We next turn to the Fourier-Jacobi case.  In the sequel we keep all the previous notations in the Bessel case, but view $V_n$ as a skew-Hermitian space by abuse of notation, which gives the unitary group  $\UU_n(\Fq)$ (up to isomorphism)  as well. Consider pairs of skew-Hermitian spaces $V_n\supset V_{n-2\ell}$ and partitions
\begin{equation}\label{pl'}
\UUl{p}'_\ell=[2\ell, 1^{n-2\ell}],\quad 0\leq \ell\leq n/2.
\end{equation}
Note that if we let $P_\ell$ be the parabolic subgroup of $\UU_n$ stabilizing $X_\ell$ and let $N_\ell$ be its unipotent radical, then $N_{\UUl{p}_\ell}=U_{\GG_\ell}\ltimes N_\ell$. Fix a nontrivial additive character $\psi$ of $\Fq$, and let $\omega_\psi$ be the Weil representation (see \cite{Ger}) of $\UU_{n-2\ell}(\Fq)\ltimes \mathcal{H}_{n-2\ell}$, where $\mathcal{H}_{n-2\ell}$ is the Heisenberg group of $V_{n-2\ell}$. Roughly speaking, there is a natural homomorphism $N_\ell(\Fq)\to \mathcal{H}_{n-2\ell}$ invariant under the conjugation action of $U_{\GG_\ell}(\Fq)$ on $N_\ell(\Fq)$, which enables us to view $\omega_\psi$ as a representation of $\UU_{n-2\ell}(\Fq)\ltimes N_{\UUl{p}_\ell}(\Fq)$. Let $\psi_\ell$ be the character of $U_{\GG_\ell}(\Fq)$ given by
\[
\psi_\ell(z)=\psi\circ\rm{Tr}_{\mathbb{F}_{q^2}/\Fq}\left(\sum^{\ell-1}_{i=1}z_{i,i+1}\right),\quad z\in U_{\GG_\ell}(\Fq).
\]
For the Fourier-Jacobi case, put
\begin{equation}\label{hnu'}
H=\UU_{n-2\ell}\ltimes N_{\UUl{p}_\ell},\quad \nu=\omega_\psi\otimes\psi_\ell.
\end{equation}
For irreducible representations $\pi$ and $\pi'$ of $\rm{U}_n(\Fq)$ and $\rm{U}_{n-2\ell}(\Fq)$ respectively, the uniqueness of {\sl Fourier-Jacobi models} asserts that
 \[
m(\pi,\pi'):=\rm{Hom}_{H(\Fq)}(\pi\otimes \bar{\nu}, \pi')\leq 1.
 \]
This was proved over $p$-adic local fields in \cite{Su}. We observe that
\[
\rm{Hom}_{H(\Fq)}(\pi\otimes \bar{\nu}, \pi')\cong \rm{Hom}_{\UU_{n-2\ell}(\Fq)}(\CJ'_\ell(\pi\otimes\bar{\omega}_\psi), \pi'),
\]
where $\CJ'_\ell(\pi\otimes\bar{\omega})$ is the twisted Jacquet module of $\pi\otimes\bar{\omega}_\psi$ with respect to $(N_{\UUl{p}_\ell}(\Fq), \psi_\ell)$. Define the $\ell$-th {\sl Fourier-Jacobi quotient} of $\pi$ to be
\begin{equation}\label{lfd}
\CQ_\ell^\rm{FJ}(\pi):=\CJ'_\ell(\pi\otimes\bar{\omega}),
\end{equation}
viewed as a representation of $\UU_{n-2\ell}(\Fq)$. Define the {\sl first occurrence index} $\ell_0:=\ell_0^\rm{FJ}(\pi)$ of $\pi$ in the Fourier-Jacobi case to be the largest nonnegative integer $\ell_0\leq n/2$ such that $\CQ^\rm{FJ}_{\ell_0}(\pi)\neq 0$. The $\ell_0$-th Fourier-Jacobi descent of $\pi$ is called the {\sl first Fourier-Jacobi descent} of $\pi$ or simply the {\sl Fourier-Jacobi descent} of $\pi$, denoted
\begin{equation}\label{bd}
\CD^\rm{FJ}_{\ell_0}(\pi):=\CQ^\rm{FJ}_{\ell_0}(\pi).
\end{equation}

Recall from \cite{LS} that irreducible unipotent representations of $\UU_n(\Fq)$ are parameterized by irreducible representations of $S_n$. It is well-known that  the latter are  parameterized by partitions of $n$.   For a partition $\lambda$ of $n$, denote by $\pi_\lambda$ the  corresponding unipotent representation of $\UU_n(\Fq)$. As is standard, we realize partitions  as Young diagrams. Then our main result is the following.

\begin{theorem}\label{main}
Let $\lambda$ be a partition of $n$ into $k$ rows, and $\lambda'$ be the partition of $n-k$ obtained by removing the first column of $\lambda$. Then the following hold.

(i) $\ell_0^\rm{B}(\pi_\lambda)\leq \frac{k-1}{2}$, with equality hold if $k$ is odd, in which case
$
\CD^\rm{B}_{\ell_0}(\pi_\lambda)=\pi_{\lambda'}.
$

(ii) $\ell_0^\rm{FJ}(\pi_\lambda)\leq \frac{k}{2}$, with equality hold if $k$ is even, in which case
$
\CD^\rm{FJ}_{\ell_0}(\pi_\lambda)=\pi_{\lambda'}.
$
\end{theorem}

Our result does not address the Bessel descent of $\pi_\lambda$ for $k$ even, nor the Fourier-Jacobi descent for $k$ odd. However, it is sufficient for some further applications (e.g. towards the wavefront set, cf. \cite[Conjecture 1.8]{JZ1}), noting that
\[
k=\max\{2\ell_0^\rm{B}(\pi_\lambda)+1, 2\ell_0^\rm{FJ}(\pi_\lambda)\}.
\]
As a special case, recall that $\UU_n(\Fq)$, $n=k(k+1)/2$ are the only unitary groups that possess unipotent cuspidal representations, and each of them has a unique unipotent cuspidal representation $\pi_k$, which corresponds to the partition $[k, k-1,\ldots, 1]$. Then we obtain the following immediate consequence, which was communicated to us by L. Zhang.

\begin{corollary} Let $\pi_k$ be the unique unipotent cuspidal representation of $\UU_n(\Fq)$, $n=k(k+1)/2$. Then the following hold.

(i) If $k$ is odd, then $\ell_0^\rm{B}(\pi_k)=\frac{k-1}{2}$ and  $\CD^\rm{B}_{\ell_0}(\pi_k)=\pi_{k-1}.$

(ii) If $k$ is even, then $\ell_0^\rm{FJ}(\pi_k)=\frac{k}{2}$ and $\CD^\rm{FJ}_{\ell_0}(\pi_k)=\pi_{k-1}.$
\end{corollary}

We will prove the equivalent forms of Theorem \ref{main} for the Bessel case and Fourier-Jacobi cases in Theorem \ref{th1} and Theorem \ref{th2} respectively.
Let us outline the strategy of the proof. First of all, Proposition \ref{prop6.2} and Proposition \ref{7.3} show that parabolic induction preserves multiplicities, which are finite field analogs of Theorem 15.1 and Theorem 16.1 in \cite{GGP1} respectively for unipotent representations. This reduces the calculation to the basic case. For Bessel models,  in order to compute the right hand side of the equality
\[
m(\pi,\pi')=\langle \pi\otimes\bar{\nu}, \pi'\rangle _{H(\Fq)}=\langle I^{\UU_{n+1}}_{P}(\tau\otimes\pi'),\pi\rangle _{\UU_{n}(\Fq)}
\]
in Proposition \ref{prop6.2}, we first extend Reeder's multiplicity formula in \cite{R} for Deligne-Lusztig representations from connected simple algebraic groups to unitary groups.  This is our main tool. However, explicit calculation with Reeder's formula is still quite involved. It will be accomplished in Section \ref{sec5}, which is the most technical part of this paper. Following  \cite{GI, Ato}, the Bessel case and Fourier-Jacobi case of the Gan-Gross-Prasad conjecture are related via theta correspondence.  By see-saw arguments, we are able to resolve the descent problem for  the latter from the former,  using explicit theta correspondence between unipotent representations given in \cite{AMR}.

A few remarks are in order.

\begin{itemize}

\item A further delicate application of theta correspondence actually resolves the Gan-Gross-Prasad conjecture for unipotent representations of finite unitary groups.  We have decided to postpone this part to a subsequent paper due to its different flavor.  In another paper we will also study the Gan-Gross-Prasad conjecture and descent problem for unipotent cuspidal representations of finite orthogonal and symplectic groups. Overall, theta correspondence and Reeder's formula will be our main tools.

\item Using the method developed in \cite{GRS}, by composition of descents one should be able to read the wavefront set of representations of finite classical groups (cf. \cite[Conjecture 1.8]{JZ1} for the $p$-adic case). For unitary groups, the nilpotent orbits therein can be recovered from the largest parts of $\UUl{p}_\ell$ and $\UUl{p}_\ell'$ given by (\ref{pl}) and (\ref{pl'}) respectively at first occurrence indices of the descents. For instance, the wavefront set of a unipotent representation $\pi_\lambda$ of $\UU_n(\Fq)$ is the singleton $\{\lambda\}$.

\item In an ongoing joint work of the first named author with D. Jiang and L. Zhang, local descents are established for classical groups over local fields (cf. \cite{JLZ}). In the $p$-adic case the spectral decomposition has not been fully understood yet. One ends up with a multiplicity free direct sum of square integrable representations whose Langlands parameters are not completely explicit. On the other hand, supercuspidal representations are compactly induced from certain cuspidal datum (see e.g. \cite{Y}), which involve representations of finite Lie groups. Our results should have applications towards explicit local descents of supercuspidal representations.
\end{itemize}

We hope to address the above issues in our future works. This paper is organized as follows. In Section \ref{sec2} we recall the description of $F$-stable maximal tori, and introduce the map from the set of $G^F$-conjugacy classes of $F$-stable maximal tori of $G$ to that of a subgroup of $G$ which we need in Reeder's formula. In Section \ref{sec3}, we briefly recall the theory of Deligne-Lusztig characters.  In Section \ref{sec4} we extend Reeder's formula to the case that $G^F=\UU_n(\Fq)$. In Section \ref{sec5} we compute the multiplicities of Deligne-Lusztig representations of $\UU_n(\Fq)$, and thereby prove the Bessel case of our main theorem using the results in Section \ref{sec4}. In Section \ref{sec6} we deduce the Fourier-Jacobi case from the Bessel case using the standard method of theta correspondence and see-saw dual pair.

\subsection*{Acknowledgement} This paper was inspired by the first named author's joint work with Dihua Jiang and Lei Zhang on the local descent. He would like to thank the latter for suggesting the descent problem over finite fields, and informing the result in the special case of unipotent cuspidal representations. He also thanks Jiajun Ma for several helpful discussions about representation theory of finite Lie groups.

\section{Remarks on maximal tori} \label{sec2}

The aim of this section is to recall the map $j_{G_s}$ used in the proof of Reeder's multiplicity formula \cite{R}.  Moreover we will calculate the map $j_{G_s}$ for $G=\UU_n$, and show that it is injective  for certain semisimple elements $s\in G^F$ of our concern.

\subsection{Construction of $j_{G_s}$}
Let $G$ be a connected reductive algebraic group over $\Fq$, $F$ be the corresponding Frobenius endomorphism, and $G^F$ be the group of rational points.  Let $T$ be an $F$-stable maximal torus in $G$. Denote its normalizer in $G$ by $N_G(T)$ and the Weyl group $W_G(T) = N_G(T)/T$. Then
$W_G(T)^F = N_G(T)^F/T^F$
by the Lang-Steinberg theorem.

Recall the classification of $F$-stable maximal tori in $G$. Fix an $F$-stable maximal torus $T_0$ in $G$ which is contained in an $F$-stable Borel subgroup of $G$, and for abbreviation write $N_G = N_G(T_0)$, $W_G = W_G(T_0)$. For any $F$-stable maximal torus $T$, there is $g\in G$ such that $^gT=T_0$. Since $T$ is $F$-stable, we have $gF(g^{-1})\in N_G$. If $w$ is the image of  $gF(g^{-1})$ in $W_G$, then we denote $T$ by $T_w$.

For a finite group $A$ with  an $F$-action,  $H^1(F, A)$ denotes the set of $F$-conjugacy classes in $A$. For $a, b\in A$, we say that $a$ is $F$-conjugate to $b$ if there is $c\in A$ such that $caF(c^{-1})=b$. Let $[a] \in H^1(F, A)$ denote the $F$-conjugacy class of an element $a \in A$.

If $w:=gF(g^{-1})T_0 \in N_G$ as above, then we obtain a class associate to $T$ given by
\[
\mathrm{cl}(T, G) := [w] \in H^1(F, W_G).
\]
It is easy to check that the $F$-conjugacy class of $w$ is independent of the choice of $g$. Hence $\mathrm{cl}(T, G)$ is well-defined. If two $F$-stable maximal tori $T_1$ and $T_2$ are $G^F$-conjugate, then $\mathrm{cl}(T_1, G)=\mathrm{cl}(T_2, G)$.

Let $s \in G^F$ be semisimple, $T_s$ be an $F$-stable maximal torus of $G_s$ contained in an $F$-stable Borel subgroup of $G_s$, and $W_{G_s}$ be the Weyl group of $T_s$ in $G_s$. Let $G_s := C_G(s)^\circ$ is the identity component of the centralizer $C_G(s)$ of $s$ in $G$.
Then we have a natural map for the set of the $G_s^F$-conjugacy classes of $F$-stable maximal tori of $G_s$ to the set of the $G^F$-conjugacy classes of $F$-stable maximal tori of $G$ by sending a torus in $G_s$ to the same torus in $G$. This induces a map
 \[j_{G_s} : H^1(F, W_{G_s}) \to H^1(F, W_G),\]
 and it follows that
\begin{equation}
\mathrm{cl}(T, G)=j_{G_s}(\mathrm{cl}(T, G_s)).\label{2.1}
\end{equation}

\subsection{The map $j_{G_s}$ for $\UU_n$}\label{2.2}
In this subsection we assume that $G=\UU_n$. We may choose the hermitian form that defines the unitary group to be represented by the identity matrix, so that
\[
\UU_n\fq =\{g \in \GGL_n(\overline{\bb{F}}_q) | g \cdot {}^t\bar{g}=I_n\},
\]
where $\bar{g}$ is obtained from raising matrix entries in $g$ to the $q$-th power.
Then the action of the Frobenius
endomorphism $F$ is given by
$
g\mapsto {}^t\bar{g}^{-1}.
$
 We shall calculate $j_{G_s}$ explicitly in this case.

Take an $F$-stable maximal torus $T_0$ of $G$ that is contained in an $F$-stable Borel subgroup. In the notation of the previous subsection, we have $W_G=W_G(T_0)\cong S_n$, and it is easy to check that under this identification  the action of $F$ on $W_G$ is given by
\[
F(w)=(w_0ww_0),
\]
where $w_0$ is the longest element in $S_n$.
Note that $w$ and $w'$ are $F$-conjugate is equivalent to that $ww_0$ and $w'w_0$ are conjuate. In particular the $F$-conjugacy classes in $W_G$ are parametrized by partitions of $n$.

\begin{lemma}\label{inj}
Assume that $G=\UU_n$. Let $s\in G^F$ be a semisimple element which is conjugate to $\mathrm{diag}(s', I_{n-m})$, where $s'$ is a regular semisimple element of $\UU_{n-m}(\mathbb{F}_q)$ such that 1 is not an eigenvalue of $s'$. Then the map $j_{G_s}$ from $H^1(F,W_{G_s})$ to $H^1(F,W_{G})$ is injective.
\end{lemma}

\begin{proof}
We have
\[
G_s=T'\times \UU_m,
 \]
 where $T'$ is the centralizer of $s'$ in $\UU_{n-m}$.
 It follows that the set $H^1(F,W_{G_s})$ is in bijection with partitions of $m$. Suppose that $T'$  corresponds to a partition $\lambda$ of $n-m$. For a partition $\mu$ of $m$, let $T_{\mu}$ be the $F$-stable maximal torus of $\UU_m$ corresponding to $\mu$, and let $T=T'\times T_{\mu}\subset G_s$. Then
  \[
 j_{G_s}(\mathrm{cl}(T,G_s))=[w_{[\lambda,\mu]}]
  \]
  where $[w_{[\lambda,\mu]}]$ is the $F$-conjugacy class in $W_G$ corresponding to the partition $[\lambda,\mu]$ of $n$. Hence
  \[
  j_{G_s}(\mathrm{cl}(T'\times T_{\mu},G_s))=j_{G_s}(\mathrm{cl}(T'\times T_{\mu'},G_s))\iff \mu=\mu'.
  \]
\end{proof}

\section{Deligne-Lusztig characters} \label{sec3}
Let $G$ be a connected reductive algebraic
group over $\mathbb{F}_q$. In their celebrated paper \cite{DL}, P. Deligne and G. Lusztig defined a virtual representation $R^{G}_{T,\theta}$ of $G^F$, associated to a character $\theta$ of $T^F$. The construction of Deligne-Lusztig characters makes use of the theory of $\ell$-adic cohomology. Here we only review some standard facts on these characters which will be used in this paper (cf. \cite[Chapter 7]{C}).

If $T$ is contained in an $F$-stable Borel subgroup $B$ of $G$, then $\theta$ lifts to a character of $B^F$ and in this case
\[
R^{G}_{T,\theta}=\mathrm{Ind}_{B^F}^{G^F}\theta.
\]
 In general, if $y = su$ is the Jordan decomposition of an element $y \in G^F$, then
\begin{equation}\label{dl}
R^{G}_{T,\theta}(y)=\frac{1}{|C^{0}(s)^F|}\sum_{g\in G^F,s^g \in T^F}\theta(s^{g})Q^{C^{0}(s)}_{{}^{g}T}(u)
\end{equation}
where $C^{0}(s) = C^{0}_{G}(s)$ is the connected component of the centralizer of $s$ in $G$, and $Q^{C^{0}(s)}_{{}^{g}T}(u)= R^{ C^{0}(s)}_{{}^{g}T,1}(u)$ is the Green function of $C^{0}(s)$ associated to ${}^{g}T$.

For a semisimple element $s\in G^F$,  put \[N_G(s,T)=\{\gamma\in G :s^{\gamma}\in T\}.\] The group $G^F_s$ acts on $N_G(s, T)^F$ by left multiplication, and we set
\[
 \bar{N}_G(s, T)^F := G^F_s   \backslash N_G(s, T)^F.
 \]
By \cite[Corollary 2.3]{R}, we have an explicit formula for $|\bar{N}_G(s, T)^F |$:
 \begin{lemma}\label{3.0}
 Let $\omega \in H^1(F, W_G)$, and $T$ be an $F$-stable maximal torus corresponding to $\omega$. Then the set $N_G(s, T)^F$ is nonempty if and only if the fiber $j^{-1}_{G_s}(\omega)$ is nonempty, in which case we have
 \[
 |\bar{N}_G(s, T)^F |=\sum_{\omega'\in j^{-1}_{G_s}(\omega)}\frac{|W_G(T)^F|}{|W_{G_s}(T_{\omega'})^F|},
 \]
 where for each $\omega'\in j^{-1}_{G_s}(\omega)$, the torus $T_{\omega'}$ is an $F$-stable maximal torus corresponding to $\omega'$.
\end{lemma}

Thus, we may rewrite (\ref{dl}) as
\[
R^{G}_{T,\theta}(y)=\sum_{\bar{\gamma}\in \bar{N}_G(s, T)^F}\theta(s^{\bar{\gamma}})Q^{G_s}_{^{\gamma}T}(u).
\]
In our future computations with $R^G_{T,\theta}$ in Reeder's formula, it will be useful to let $s$ vary in $G^F$ in such a way that $G_s$ is unchanged. Let $Z(G_s)$ denote the center of $G_s$. For $\omega'\in j^{-1}_{G_s}(\omega)$,
the function
\[
\theta_{\omega'}(z):=\sum_{\bar{\gamma}\in\mathcal{O}_{\omega'}}\theta(z^{\bar{\gamma}})
\]
is well-defined on $Z(G_s)^F$, where $\mathcal{O}_{\omega'}$ is the $W_G(T)^F$-orbit in $\bar{N}_G(s, T)^F $ corresponding to $\omega'\in j^{-1}_{G_s}(\omega)$ as in Lemma \ref{3.0}. Then we have
\[
R^{G}_{T,\theta}(zu)=\sum_{\omega'\in j^{-1}_{G_s}(\omega)}\theta_{\omega'}(z)Q^{G_s}_{^{\gamma}T}(u),\quad \mathrm{if}\ G_z=G_s.
\]

More generally, for an $F$-stable Levi subgroup $L$ which is not necessarily contained in an $F$-stable parabolic subgroup, and a representation $\pi$ of $L^F$, one has the generalized Deligne-Lusztig induction $R^G_{L}(\pi)$. In particular, if $L$ is contained in an $F$-stable parabolic subgroup $P$, then $R^G_{L}(\pi)$ coincides with the parabolic induction
\begin{equation}\label{paraind}
I_P^G(\pi):=\mathrm{Ind}_{P^F}^{G^F}\pi.
\end{equation}
For instance, recall that if $L=T$ is contained in an $F$-stable Borel subgroup $B$, then
\[
I_B^G(\theta)=R^G_{T,\theta}.
\]
The following result is standard.

\begin{proposition} \label {3.1}
Let $Q \subset P$ be $F$-stable parabolic subgroups of $G$, and $M \subset L$ be their Levi subgroups respectively. Then
\[
I^{G}_{P}\circ I^{L}_{L\cap Q}=I^{G}_{Q}.
\]
\end{proposition}

An $F$-stable maximal torus $T$ is said to be minisotropic if $T$ is not contained in any $F$-stable proper parabolic subgroup of $G$. Then a representation $\pi$ of $G^F$ is cuspidal if and only if
\[
\langle \pi,R^G_{T,\theta}\rangle_{G^F}=0
\]
whenever $T$ is not minisotropic, for any character $\theta$ of $T^F$ (see \cite[Theorem 6.25]{S1}). Note that if $G=\GGL_n$, then $T$ is  minisotropic if and only if $T^F\cong \GGL_1(\mathbb{F}_{q^n})$.

Let $\theta\in \hat{T}^F$, $\theta'\in \hat{T'}^F$ where $T$, $T'$ are $F$-stable maximal tori. Two pairs $(T,\theta)$ and $(T',\theta')$ are said to be geometrically conjugate if for some $n\ge 1$, there exists $x \in G^{F^n}$ such that
\[
^xT^{F^n} = T ^{\prime F^n}\ \mathrm{and}\ \ ^x(\theta \circ  N) = \theta'\circ N
 \]
 where $N$ is the norm map. An irreducible character of $G^F$ is called a regular character if it appears in the Gelfand-Graev character (c.f. \cite[p. 281]{C}). By \cite[p. 378]{C}, in any geometric conjugacy class $\kappa$ of the pairs $(T,\theta)$, there is only one regular character $\pi^{reg}_\kappa$ that appears in  $R^G_{T,\theta}$ for some $(T,\theta)\in \kappa$; and any regular character only appears in one geometric conjugacy class. Moreover
\[
\pi^{reg}_\kappa=\sum_{(T,\theta)\in \kappa
\ \rm{mod} \  G^F}\frac{\varepsilon_G \varepsilon_T R^G_{T,\theta}}{\langle R^G_{T,\theta},R^G_{T,\theta}\rangle _{G^F}}.
\]
The above equation implies that $\pi_\kappa^{reg}$ appears in $R^G_{T,\theta}$ for each pair $(T,\theta)\in \kappa$. Thus $\pi_\kappa^{reg}$ is cuspidal if and only if $T$ is minisotropic and $\theta$ is regular for every pair $(T,\theta)\in \kappa$.  Here $\theta$ regular means that
\[
^x\theta=\theta, \ x\in W_G(T)^F \ \mathrm{if \ and \ only\ if\ } x=1.
\]
In particular, if $\tau$ is an irreducible cuspidal representation of $\GGL_n(\Fq)$, then there is a pair $(T,\theta)$ with $T$ an $F$-stable minisotropic maximal torus and $\theta$ regular such that
$
\tau=\pm R_{T,\theta}^G.
$

We now review the classification of representations of $\UU_n(\Fq)$. The classification of the representations of $\UU_n(\Fq)$ was given by Lusztig and Srinivasan in \cite{LS}. As in Subsection \ref{2.2}, for $G=\UU_n$ we may and do identity $W_G$ with $S_n$, and we let $T_w$ be the $F$-stable maximal torus of $\UU_n$ corresponding to $w\in S_n$.

\begin{theorem}\label{thm3.3}
Let $\sigma$ be the character of an irreducible representation of $S_n$. Then
\[
R_\sigma^{\UU_n}:=\frac{1}{|S_n|}\sum_{w\in S_n}\sigma(ww_0)R_{T_w,1}^{\UU_n}
\]
is (up to sign) a unipotent representation of $\UU_n(\Fq)$ and all unipotent representations of $\UU_n(\Fq)$ arise in this way.\end{theorem}

By the above theorem, we know that every irreducible representation of $S_n$ corresponds to a unipotent representation of $\UU_n(\Fq)$. It is well known that irreducible representations of $S_n$ are parametrized by partitions of $n$. For a partition $\lambda$ of $n$,  denote by  $\sigma_\lambda$ the corresponding representation of $S_n$, and
let $R_\lambda^{\UU_n}:=R_{\sigma_\lambda}^{\UU_n}$ be the corresponding unipotent representation of $\UU_n(\Fq)$.
By Lusztig's result \cite{L2},  $R_\lambda^{\UU_n}$ is (up to sign) a unipotent cuspidal representation  of $\UU_n(\Fq)$ if and only if $n=\frac{k(k+1)}{2}$ for some positive integer $k$ and  $\lambda=[k,k-1,\cdots,1]$.

\section{Reeder's formula} \label{sec4}
As before,  let $G$ be a connected reductive algebraic
group over $\mathbb{F}_q$. Let $H \subset G$ be a connected reductive subgroup of $G$ over $\Fq$, and $S$ be an $F$-stable maximal torus of $H$. Let $B$ and $B_H$ be Borel subgroups of $G$ and $H$, respectively, and let $\delta$ be the minimum codimension of a $B_H$-orbit in $G/B$.

In \cite{R}, Reeder gives a formula  for the multiplicity $\langle R_{T,\theta}^G,R_{S,\theta'}^H\rangle _{H^F}$ when $G$ and $H$ are simple. More precisely, by \cite [Theorem 1.4]{R} there is a polynomial $M(t)$ of degree at most $\delta$ whose coefficients depend on the characters $\theta$ and $\theta'$ of $T^F$ and $S^F$ respectively, and an integer $m \ge 1$ such that
\[
\langle R_{T,\theta^\nu}^G,R_{S,\theta^{\prime \nu}}^H\rangle _{H^{F^\nu}}=M(q^\nu)
\]
 for all positive integers $\nu \equiv 1$ mod $m$, where $\theta^\nu=\theta\circ N_\nu^T$ with $ N_\nu^T: T^{F^\nu}\to T^F$ the norm map. And the degree $\delta$ is optimal. Moreover, \cite[Proposition 7.4]{R} gives an explicit formula for the leading coefficient  in $M(t)$, which we denote by $A$ below.

 In order to calculate $\langle I^{\UU_{n+1}}_{P}(\tau\otimes\pi'),\pi_\lambda\rangle _{\UU _n(\Fq)}$ using Reeder's method, it is necessary to extend his result from connected simple algebraic groups to unitary groups. In particular, we will show that $\deg M(t)=0$, hence $\langle I^{\UU_{n+1}}_{P}(\tau\otimes\pi'),\pi_\lambda \rangle _{\UU _n(\Fq)}=A$. In this case, the explicit formula for $A$ will be given in Proposition \ref{proposition:A}.

 In the sequel, we will first recall some results in \cite[Sections 3-5]{R}  without  proof, which do not require that $G$ is simple. Then we extend  \cite[Sections 6-7]{R} to the case of $\UU_n$.

 \subsection{A partition of $S$}
 We first recall the notations in \cite{R}. Let $I(S)$ be an index set for the following set of subgroups of $G$,
 \[
 \{G_s:s\in S\}
 \]
 where $G_s:=C_{G}(s)^{\circ}$. Note that if $G=\UU_n$, then $I(S)$ is finite. For $\iota \in I(S)$, denote by $G_{\iota}$ be the corresponding connected centralizer, and put
 \[
 S_\iota:=\{s\in S:G_s=G_\iota\}.
 \]
 The $F$-action on $S$ induces a permutation of $I(S)$, and let $I(S)^F$ be the $F$-fixed points in $I(S)$. Note that if $S_\iota^F$ is nonempty, then $\iota \in  I(S)^F$.

For $\iota \in  I(S)$,  set
 \[
 H_\iota:=(H\cap G_\iota)^\circ.
 \]
We observe that if $G_s=G_\iota$, then
 \[
 s\in H_\iota\subset G_\iota.
 \]
Put $Z_\iota:=Z(G_\iota)\cap S$. Then it is easy to check that
 \[
 \begin{aligned}
 &Z_\iota\subset Z(H_\iota),\\
 &S_\iota\subset Z_\iota \subset S,\\
 &G_\iota=C_{G}(Z_\iota)^\circ.
 \end{aligned}
 \]
Define a partial ordering on $I(S)$ by
 \[
 \iota'\le\iota \quad \Longleftrightarrow \quad G_\iota \subset G_{\iota'} \quad \Longleftrightarrow \quad Z_{\iota'}\subset Z_{\iota},
 \]
 and set
 \[
 I(\iota, S) := \{\iota \in I(S) : \iota' < \iota\}.
 \]
 For a subset $J \subset I(S)$, put
 \[
 Z_J:=\bigcap_{\iota\in J}Z_\iota.
 \]

 \subsection{Multiplicity as a polynomial}
  Let $f_J(t)$ be the polynomial of degree $\dim Z_J$  in \cite[Section 5.6]{R}. Let $Q_{v,u}^{G_\iota}(t)$ and $Q_{\varsigma,u}^{H_\iota}(t) $ be the Green polynomials in \cite[Section 5.4]{R}  and $P_{\iota,u}(t)$ be
the polynomial given by \cite[(5.11)]{R}. Then we have
\[
\det P_{\iota,u}(t)=\dim C_{H_\iota}(u), \quad \deg f_{\iota}(t)=\dim Z_\iota.
\]
We set
\begin{equation}\label{chi}
\chi_v:=\frac{1}{|W_{G_\iota}(T)^F|}\sum_{x\in W_{G}(T) }(^x\chi)|_{Z_\iota}\quad \mathrm{and}\quad \eta_\varsigma:=\frac{1}{|W_{H_\iota}(S)^F|}\sum_{y\in W_{H}(S) }(^y\eta)|_{Z_\iota}.
\end{equation}
Denote the set of unipotent elements of $H_\iota$ by $\mathcal{U}(H_\iota)$. Let $\alpha$ denote a quadruple  $(\iota,u,v,\varsigma)$, where
  \begin{equation}\label{q}
  \iota\in I(S)^F,\quad [u]\in\mathcal{U}(H_\iota)^F,\quad v\in j^{-1}_{G_\iota}(\mathrm{cl}(T,G)), \quad \varsigma\in j^{-1}_{H_\iota}(\mathrm{cl}(S,H)).
  \end{equation}
 Let $\Theta_{\alpha}(t)$ be the rational function given by
  \begin{equation}\label{theta}
  \Theta_{\alpha}(t):=\langle \chi_v, \eta_\varsigma\rangle _{Z_\iota^F}+\sum_{J\subset I(\iota,S)^F}(-1)^{|J|}\langle \chi_v, \eta_\varsigma\rangle _{Z_J^F}\frac{f_J(t)}{f_{\iota}(t)},
 \end{equation}
  and $\Psi_{\alpha}(t)$ be the rational function given by
\begin{equation}\label{psi}
  \Psi_{\alpha}(t):=f_{\iota}(t) \cdot \frac{\overline{Q_{v,u}^{G_\iota}(t) }Q_{\varsigma,u}^{H_\iota}(t)}{|\bar{N}_{H}(\iota,S)^F||P_{\iota,u}(t)|}.
 \end{equation}
 The following result is given by \cite[(5.18)]{R}:
 \[
 \langle R^G_{T,\chi},R^H_{S,\eta}\rangle _{H^F}=\sum_{\alpha}\Psi_{\alpha}(q)\Theta_{\alpha}(q),
 \]
 where $\alpha$ runs over the quadruples  $(\iota,u,v,\varsigma)$ given by (\ref{q}).

\subsection{Degree of $\Psi_{\alpha}(t)$}
From now on, assume that
\begin{equation} \label{gh}
(G, H)=(\UU_{n+1}, \UU_n).
\end{equation}
Let $u$ be a unipotent element given by (\ref{q}).
Let $\mathcal{B}_G$ be the variety of Borel subgroups of $G$, and let $\mathcal{B}_G^u$ be the variety of fixed points of $u$. Set
\[
d_G(u) := \dim  \mathcal{B}_G^u.
\]
Steinberg proved that
\[
2d_G(u) = \dim C_G(u) - \overline{rk} \ G,
\]
where $\overline{rk} \ G$ is the absolute rank of $G$.

From Section 5.4 and equation (5.9) in \cite{R}, we find that
\begin{align}
\label{equ:deg1} \deg \Psi_{\alpha}(t) & \le\dim Z_\iota+d_{G_\iota}(u)+d_{H_\iota}(u)-\dim C_{H_\iota}(u)\\
&=\dim Z_\iota+\frac{1}{2}(\dim C_{G_\iota}(u)-\dim C_{H_\iota}(u)-\bar{rk} \ G-\bar{rk} \ H)\nonumber.
\end{align}
Define
\[
\begin{aligned}
\delta_\iota:&=\dim Z_\iota+\dim \mathcal{B}_{G_\iota}-\dim \mathcal{B}_{H_\iota}-\dim S\\
&=\dim Z_\iota+\frac{1}{2}(\dim C_{G'_\iota}(1)-\dim C_{H'_\iota}(1)-\overline{rk} \ G-\overline{rk} \ H).
\end{aligned}
\]

\begin{lemma} \label{lem:deg1}
Assume (\ref{gh}). Then for a quadruple $\alpha =(\iota,u,v,\varsigma)$ given by (\ref{q}), one has $\deg \Psi_{\alpha}(t) \le \delta_\iota$ with equality hold only if $u = 1$.
\end{lemma}
\begin{proof}
By \cite[(6.3)]{R}, for a simple group $G'$ we have
\begin{equation}
\dim C_{G'_\iota}(u)-\dim C_{H'_\iota}(u)\le\dim C_{G'_\iota}(1)-\dim C_{H'_\iota}(1).\label{equ:deg2}
\end{equation}
Since $\rm{SU}_{n+1}$ is simple, it follows that
\begin{align} \label{equ:deg3}
& \dim C_{\UU_{n+1,\iota}}(u)-\dim C_{\UU_{n,\iota}}(u) \\
=&
\dim C_{\rm{SU}_{n+1,\iota}}(u)+1-(\dim C_{\rm{SU}_{n,\iota}}(u)+1) \nonumber \\
=&
\dim C_{\rm{SU}_{n+1,\iota}}(u)-\dim C_{\rm{SU}_{n,\iota}}(u) \nonumber \\
\le& \dim C_{\rm{SU}_{n+1,\iota}}(1)-\dim C_{\rm{SU}_{n, \iota}}(1) \nonumber \\
=& \dim C_{\UU_{n+1, \iota}}(1)-\dim C_{\UU_{n, \iota}}(1).\nonumber
\end{align}
Thus (\ref{equ:deg2}) also holds for $(\UU_{n+1}, \UU_n)$. Then the desired inequality follows from (\ref{equ:deg1}). Since equality holds in (\ref{equ:deg1}) for $(\rm{SU}_{n+1}, \rm{SU}_n)$  if and only if $u=1$  (see \cite[lemma 6.4]{R}), it follows easily that equality holds in (\ref{equ:deg2}) for $(\UU_{n+1}, \UU_n)$ if and only if $u=1$ as well.
\end{proof}

We further have:

\begin{lemma} \label{lem:deg2}
Assume (\ref{gh}).  Then $\deg \Psi_{\alpha}(t) \le 0$, with  equality hold only if $u = 1$ and there is $s\in S_\iota$ conjugate to $\rm{diag}(s', I_{m})$, where $s'$ is a regular semisimple element of $\UU_{n-m}(\bb{F}_q)$ such that $1$ is not an eigenvalue of $s'$.
\end{lemma}

\begin{proof}
For a fixed $\iota$, let $s\in S_\iota$ be a semisimple element in $H^F$. Assume that $s$ has $r$ distinct eigenvalues which are not equal to $1$. Let $m$ be the multiplicity of the eigenvalue $1$ of $s$. Then $s$ conjugate to $\rm{diag}(s', I_{m})$, where $s'$ is a semisimple element of $\UU_{n-m}(\bb{F}_q)$ such that $1$ is not an eigenvalue of $s'$. Thus, $G_s$ and $H_s$ are isomorphic to $C_{\UU_{n-m}}(s')\times \UU_{m+1}$ and
$C_{\UU_{n-m}}(s')\times\UU_{m}$ respectively.
It follows that \[\dim \mathcal{B}_{G_\iota}-\dim \mathcal{B}_{H_\iota}=\dim \mathcal{B}_{ \UU_{m+1}}-\dim \mathcal{B}_{ \UU_{m}}=m.\]

We obtain that $\delta_\iota=\dim Z_\iota-n+m\le 0$, with equality hold only if $r=n-m$. In other words, $\delta_\iota=0$ if and only if $s'$ is a regular semisimple element of $\UU_{n-m}(\bb{F}_q)$ such that $1$ is not an eigenvalue of $s'$.
\end{proof}

\subsection{The leading term of $\Theta_{\alpha}(q)$}
We have known that $\deg \Theta_{\alpha}(t)\le 0$ (see \cite[(5.17)]{R}). We now show that in fact there is only one term in $\Theta_{\alpha}(q)$ of degree zero.

\begin{lemma}
Assume (\ref{gh}), $\iota'<\iota$ and that $\delta_\iota=0$. Then $\dim Z_{\iota'}<\dim Z_\iota$.
\end{lemma}
\begin{proof}
Let $s'\in S_{\iota'}$.
Let $n_1,\ldots, n_r$ be the multiplicities of distinct eigenvalues of $s'$ which are not equal to $1$, and $m$ be the multiplicity of the eigenvalue $1$, so that $n_1+\cdots+n_r+m=n$. Then $\dim Z_{\iota'}=\dim Z(G_{\iota'})\cap S =r$.

Since $\iota'<\iota$, we have $G_\iota\subset G_{\iota'}$. Thus, $s\in S_{\iota}$ is a semisimple element of $G_{\iota'}$. Since $\delta_\iota=0$, by Lemma \ref{lem:deg2}, $s$ is conjugate to $\rm{diag}(s', I_{m})$, where $s'$ is a regular semisimple element of $\UU_{n-m}(\bb{F}_q)$ such that $1$ is not an eigenvalue of $s'$. It follows that $G_s\cong T\times \UU_{n+1-m}$ and $\dim Z_\iota=n-m\ge r$, where $T$ is an $F$-stable maximal torus of $\UU_{n-m}$. It is easy to check that
$\dim Z_{\iota}=r$ if and only if $n_i=1$, $i=1,\ldots,r$. However in this case $\iota=\iota'$.
\end{proof}
Recall that $f_J(t)$ is a polynomial of degree $\dim  Z_J$. The above lemma implies that $\deg f_J(t)<\deg f_{\iota}(t)$. Recall the formula (\ref{theta})
 \[
  \Theta_{\alpha}(t)=\langle \chi_v, \eta_\varsigma\rangle _{Z_\iota^F}+\sum_{J\subset I(\iota,S)^F}(-1)^{|J|}\langle \chi_v, \eta_\varsigma\rangle _{Z_J^F}\frac{f_J(t)}{f_{\iota}(t)}.
  \]
Then  the only degree zero term of $\Theta_{\alpha}(q)$  is $\langle \chi_v, \eta_\varsigma\rangle _{Z_\iota^F}$.

\subsection{The leading term of $M(t)$}
We shall find an explicit and effective formula of $M(t)$. Recall that
\[
 \langle R^G_{T,\chi},R^H_{S,\eta}\rangle _{H^F}=\sum_{\alpha}\Psi_{\alpha}(q)\Theta_{\alpha}(q).
 \]
Assuming (\ref{gh}), we have known that $\deg \Psi_{\alpha}(t)\le 0$ and $\deg \Theta_{\alpha}(t)\le 0$. So the polynomial $M(t)$ is a constant, which equals the coefficient of the leading term of $\sum_\alpha \Psi_{\alpha}(t) \Theta_{\alpha}(t)$.

 We now calculate $M(t)$ by the definitions of $\Psi_{\alpha}(t)$ and $\Theta_{\alpha}(t)$ in (\ref{theta}) and (\ref{psi}) respectively. By Lemmas \ref{lem:deg1} and Lemma \ref{lem:deg2}, only quadruples $\alpha$ with $u = 1$ and $\delta _\iota = 0$ contribute to the leading term; henceforth we assume that $\alpha$ is of this form. If we take $u = 1$, then by \cite[(5.10)]{R},
 \[
 Q_{v,1}^{G_\iota}(t)=\epsilon_G(v)t^{d_G(1)} +\mathrm{lower\ powers\ of}\ t,
 \]
 where $\epsilon_G(v)=(-1)^{\mathrm{rk}(G)+\mathrm{rk}(T_v)}$.
From \cite[Section 5.5]{R}, we know that $P_{\iota,u}(t)$ is of the form $A_\iota(u)$ times a monic polynomial in $\mathbb{Z}[t]$, where $A_\iota(u)$ is the centralizers in $H_\iota$ of all unipotent elements $u\in H_\iota^F$ for every $\iota\in I(S)^F$. By \cite[Section 5.6]{R}, for any subset $J\subset I(S)$,
 \[
 f_J(q^\nu)=|Z_J^{F^\nu}|,\quad\mathrm{for\ all\ \nu\equiv1\ mod }\ m,
 \]
 where $m$ is a positive integer divisible by the exponent of the finite group $W_G\rtimes \langle \vartheta\rangle$. Here  the actions of $\vartheta$ as automorphisms on  the set of roots $\Phi(T_0,G)$ as well as $W_G$, are induced by $F$.

 As a power series in $t$, the leading coefficient of $\Psi_{\alpha}(t)$ is
 \[
 \frac{(-1)^{\mathrm{rk}(G_\iota)+\mathrm{rk}(H_\iota)+\mathrm{rk}(T)+\mathrm{rk}(S)}}{|\bar{N}_{H}(\iota,S)^F|},
 \]
 and the leading coefficient of $\Theta_{\alpha}(t)$ is
 \[
 \langle \chi_v, \eta_\varsigma\rangle _{Z_\iota^F}.
 \]
Note that by our selection of $\iota\in I(S)^F$, we always have $j_{H_\iota}^{-1}(\mathrm{cl}(S,H))\ne\emptyset$. Thus combining the above calculations and Lemma \ref{inj}, we obtain that
 \begin{proposition}\label{proposition:A}
  Assume (\ref{gh}). Then
 \[
\langle R^G_{T,\chi},R^H_{S,\eta}\rangle _{H^F}=\sum_{\mbox{\tiny$\begin{array}{c}\iota\in I(S)^F,
\delta_\iota=0\\j_{G_\iota}^{-1}(\mathrm{cl}(T,G))\ne \emptyset
\end{array}$}}\frac{(-1)^{\mathrm{rk}(G_\iota)+\mathrm{rk}(H_\iota)+\mathrm{rk}(T)+\mathrm{rk}(S)}}{|\bar{N}_{H}(\iota,S)^F|}\langle \chi_v, \eta_\varsigma\rangle _{Z_\iota^F},
\]
where $v\in j_{G_\iota}^{-1}(\mathrm{cl}(T,G))$ and $\varsigma\in j_{H_\iota}^{-1}(\mathrm{cl}(S,H))$.
\end{proposition}

\subsection{Calculation of $\langle R^{G}_{T_1\times T_2,\theta\otimes1},R_{S,1}^{H}\rangle _{H(\Fq)}$}

Since every unipotent representation of $H^F$ is a sum of Deligne-Lusztig characters $R_{S,1}^{H}$, we only need to manipulate Reeder's formula for
\[
\langle R^{G}_{T,\theta_T},R_{S,1}^{H}\rangle _{H(\Fq)}.
\]
A pair $(T,\theta_T)$ for $\UU_{n+1}$ corresponds to a pair $(T^*, s)$, where $T^*$ is an $F$-stable maximal torus of the dual group $\UU_{n+1}^*\cong \UU_{n+1}$, and $s\in T^{*F}$. We say that $1\not\in (T,\theta_T)$ if 1 is not an eigenvalue of $s$.
Note that every pair $(T,\theta_T)$ is of the form $(T_1\times T_2,\theta\otimes 1 )$ such that $1\notin (T_1,\theta)$. We now turn to compute
\[
\langle R^{G}_{T_1\times T_2,\theta\otimes1},R_{S,1}^{H}\rangle _{H(\Fq)}.
\]

Thus we let $T=T_1\times T_2$ be an $F$-stable maximal torus of $\UU_{n+1}$, and $\theta\otimes1$ be a character of $T_1^F\times T_2^F$ with $1\notin (T_1,\theta)$. By Proposition \ref{proposition:A}, for any $F$-stable maximal torus $S$ of $H$,
\begin{equation}
\langle R^{G}_{T_1\times T_2,\theta\otimes1},R_{S,1}^{H}\rangle _{H(\Fq)}
=\sum_{\mbox{\tiny$\begin{array}{c}\iota\in I(S)^F,
\delta_\iota=0\\j_{G_\iota}^{-1}(\mathrm{cl}(T,G))\ne \emptyset
\end{array}$}}\frac{(-1)^{\mathrm{rk}(G_\iota)+\mathrm{rk}(H_\iota)+\mathrm{rk}(T)+\mathrm{rk}(S)}}{|\bar{N}_{H}(\iota,S)^F|}\langle (\theta\otimes1)_v, 1_\varsigma\rangle _{Z_\iota^F} \label{F1}
\end{equation}
where $v\in j_{G_\iota}^{-1}(\mathrm{cl}(T_1\times T_2,G))$ and $\varsigma\in j_{H_\iota}^{-1}(\mathrm{cl}(S,H))$.

 We now calculate the pairing $\langle (\theta\otimes1)_v, 1_\varsigma\rangle _{Z_\iota^F}$. We may conjugate $T_1\times T_2$ and $S$ to ensure that $Z_\iota \subset (T_1\times T_2) \cap S$. Then by (\ref{chi}),
\begin{align*}
& (\theta\otimes1)_v=\frac{1}{|W_{G_\iota}(T_1\times T_2)^F|}\sum_{x\in W_{G}(T_1\times T_2)^F }{}^x(\theta\otimes1)|_{Z_\iota},\\
& 1_\varsigma=\frac{1}{|W_{H_\iota}(S)^F|}\sum_{y\in W_{H}(S)^F}({}^y1)|_{Z_\iota}=\frac{|W_{H}(S)^F|}{|W_{H_\iota}(S)^F|}\cdot1.
\end{align*}
Recall Lemma \ref{3.0} that $|\bar{N}_{H}(\iota,S)^F|=\dfrac{|W_{H}(S)^F|}{|W_{H_\iota}(S)^F|}$. Rewrite (\ref{F1}) as
\begin{align}
\langle R^{G}_{T_1\times T_2,\theta\otimes1},R_{S,1}^{H}\rangle _{H(\Fq)}
&=\sum_{\mbox{\tiny$\begin{array}{c}\iota\in I(S)^F,
\delta_\iota=0\\j_{G_\iota}^{-1}(\mathrm{cl}(T,G))\ne \emptyset
\end{array}$}}\frac{(-1)^{\mathrm{rk}(G_\iota)+\mathrm{rk}(H_\iota)+\mathrm{rk}(T)+\mathrm{rk}(S)}}{|\bar{N}_{H}(\iota,S)^F|}\langle (\theta\otimes1)_v, 1_\varsigma\rangle _{Z_\iota^F} \label{F11}\\
&=\sum_{\mbox{\tiny$\begin{array}{c}\iota\in I(S)^F,
\delta_\iota=0\\j_{G_\iota}^{-1}(\mathrm{cl}(T,G))\ne \emptyset
\end{array}$}}(-1)^{\mathrm{rk}(G_\iota)+\mathrm{rk}(H_\iota)+\mathrm{rk}(T)+\mathrm{rk}(S)}\frac{|W_{T_1\times T_2,\theta\otimes1,\iota}|}{|W_{G_\iota}(T_1\times T_2)^F|},\nonumber
\end{align}
where
\begin{equation}\label{wt}
W_{T_1\times T_2,\theta\otimes1,\iota}:=\{ x\in W_{G}(T_1\times T_2)^F : {}^x(\theta\otimes1)|_{Z_\iota^F}=1\}.
\end{equation}

  It is easy to check that if $G_{\iota_1}^F$ and $G_{\iota_2}^F$ are in the same $H^F$-conjugacy class, then we have
\begin{align*}
& \delta_{\iota_1}=\delta_{\iota_2},\\
&  (-1)^{\mathrm{rk}(G_{\iota_1})+\mathrm{rk}(H_{\iota_1})}\frac{|W_{T_1\times T_2,\theta\otimes1,\iota_1}|}{|W_{G_{\iota_1}}(T_1\times T_2)^F|}=(-1)^{\mathrm{rk}(G_{\iota_2})+\mathrm{rk}(H_{\iota_2})}\frac{|W_{T_1\times T_2,\theta\otimes1,\iota_2}|}{|W_{G_{\iota_2}}(T_1\times T_2)^F|},\\
& j_{G_{\iota_1}}^{-1}(\mathrm{cl}(T,G))\ne\emptyset \Longleftrightarrow j_{G_{\iota_2}}^{-1}(\mathrm{cl}(T,G))\ne\emptyset \textrm{ for any }T.
\end{align*}
This means that the properties of $\iota$ involved in the sum of (\ref{F11}) do not change when $G_\iota^F$ varies over a $H^F$-conjugacy class.

  We now calculate the conjugacy classes of $G^F_\iota$. For two elements $\iota_1, \iota_2\in I(S)^F$, we say that $\iota_1\sim\iota_2$ if $G_{\iota_1}^F$ and $G_{\iota_2}^F$ are in the same $H^F$-conjugacy class. Denote the equivalence class of $\iota$ by $[\iota]$.  Let $[I(S)^F]$ be the set of equivalence classes in $I(S)^F$.

Let us denote a typical summand in (\ref{F11}) by $X_\iota$. Then  (\ref{F11})  can be rewritten as
\[
\sum_{\mbox{\tiny$\begin{array}{c}\iota\in I(S)^F,
\delta_\iota=0\\j_{G_\iota}^{-1}(\mathrm{cl}(T,G))\ne \emptyset
\end{array}$}}X_\iota=\sum_{\mbox{\tiny$\begin{array}{c}[\iota]\in [I(S)^F],
\delta_\iota=0\\j_{G_\iota}^{-1}(\mathrm{cl}(T,G))\ne \emptyset
\end{array}$}} \#[\iota]\cdot X_\iota.
\]

For two partition $\lambda$ and $\lambda'$ of $n$ and $n'$, respectively, we say that $\lambda'\subset\lambda$ if $\{\lambda_i'\}\subset \{\lambda_i\}$ and $\lambda'\nsubseteq\lambda$ otherwise. Here $\{\lambda_i'\}\subset \{\lambda_i\}$  means the containment of multisets of integers.

Let $\mu^S$ be a partition of $n$ corresponding to $S$. Recall Lemma \ref{lem:deg2} that if $\delta_\iota=0$, then there is $s\in S_\iota$ conjugate to $\rm{diag}(s', I_{m})$, where $s'$ is a regular semisimple element of $\UU_{n-m}(\bb{F}_q)$ such that $1$ is not an eigenvalue of $s'$. In this case
\begin{equation}\label{g}
G_\iota\cong G_s\cong T'\times \UU_m,
\end{equation}
where $T'$ is the centralizer of $s'$ in $\UU_{n-m}$.
Hence the set
\[
\{[\iota]\in[I(S)^F]:\delta_{\iota}=0\}
 \]
 is parameterized by pairs $(m,\mu^\iota)$, where $m\leq n$ is a nonnegative integer and $\mu^\iota\subset\mu^S$ is a partition of $n-m$ corresponding to $T'$. For a partition $\mu'$, if $\delta_{\iota}=0$ and $\mu^\iota=\mu'$ then we will denote $[\iota]$ by $[\iota_{\mu'}]$.

 Note that only the terms satisfying $j_{G_\iota}^{-1}(\mathrm{cl}(T_1\times T_2,G))\ne \emptyset$ contribute to (\ref{F1}).  The following lemma gives an explicit description of these terms, which follows from Lemma  \ref{inj} and (\ref{2.1}).

 \begin{lemma}\label{6.1}
Assume that $[\iota]\in[I(S)^F]$, $\delta_{\iota}=0$. Let $T$ be an $F$-stable maximal torus of $G$ which corresponds to a partition $\mu$ of $n+1$. Then $j_{G_\iota}^{-1}(\mathrm{cl}(T,G))\ne \emptyset$ if and only if $\mu^\iota\subset \mu$.
\end{lemma}

  From this lemma we obtain that
 \begin{equation}\label{x}
 \begin{aligned}
\langle R^{G}_{T_1\times T_2,\theta\otimes1},R_{S,1}^{H}\rangle _{H(\Fq)}
=&\sum_{\mbox{\tiny$\begin{array}{c}[\iota]\in [I(S)^F],\delta_{\iota}=0\\
j_{G_\iota}^{-1}(\mathrm{cl}(T_1\times T_2,G))\ne \emptyset \end{array}$}}\#[\iota]\cdot X_\iota
=\sum_{\mbox{\tiny$\begin{array}{c}[\iota]\in [I(S)^F],\delta_{\iota}=0\\
\mu^\iota\subset\mu\\
\end{array}$}}\#[\iota]\cdot X_{\iota}\\
=&\sum_{\mbox{\tiny$\begin{array}{c}\mu'\subset\mu\\
\mu'\subset\mu^S\\
\end{array}$}}\#[\iota_{\mu'}]\cdot X_{\iota_{\mu'}}
\end{aligned}
 \end{equation}

We now calculate $\#[\iota_{\mu'}]$. Let $\lambda$ and $\lambda'$ are two partitions and $\lambda'\subset\lambda$. Let us write
 \begin{equation}\label{pair}
 \lambda=[n_1^{a_1}, \ldots, n_l^{a_l}]\quad \rm{and}\quad \lambda'=[n_1^{b_1},\ldots, n_l^{b_l}],
 \end{equation}
where $n_i$'s are distinct so that $0\leq b_i\leq a_i$, $i=1,\ldots, l$. We set
\begin{equation} \label{CC}
C_{\lambda,\lambda'}=\prod_{i=1}^l
{a_i \choose b_i}.
\end{equation}
For example, $C_{[2,1^2],[2,1]}=2$. Note that if $\iota\sim \iota'$, then $\mu^{\iota}=\mu^{\iota'}$. It follows that
 \begin{equation}\label{co}
\#[\iota_{\mu'}]=C_{\mu^S,\mu'}.
 \end{equation}

For two partitions $\mu^1$ and $\mu^2$, denote by $\mu=[\mu^1, \mu^2]$ the unordered partition of $|\mu^1|+|\mu^2|$ formed by taking unions of the
parts in $\mu^1$ and $\mu^2$. We say that an $F$-stable maximal torus of $\UU_{n+1}$ of the form $T=T_1\times T_2$ corresponds to a partition $\mu=[\mu^1,\mu^2]$ of $n+1$, where $|\mu^1|+|\mu^2|=n+1$, if $T_i$ is an $F$-stable maximal torus of $\UU_{|\mu^i|}$ corresponding to the partition $\mu^i$, $i=1, 2$.
Then by Lemma \ref{6.1}, (\ref{x}) and (\ref{co}), we have

 \begin{proposition}\label{p1}
  Let $T=T_1\times T_2$ be an F-stable maximal torus of $\UU_{n+1}$ corresponding to a
partition $\mu=[\mu^1,\mu^2]$ of $n+1$, and $\theta\otimes1$ be a character of $T_1^F\times T_2^F$. Let $S$ be an F-stable maximal torus of $\UU_{n}$ corresponding to a
partition $\mu^S$ of $n$. Assume that $1\notin (T_1,\theta)$. Then
  \begin{equation}\label{XD}
\langle R^{G}_{T_1\times T_2,\theta\otimes1},R_{S,1}^{H}\rangle _{H(\Fq)}=
\sum_{\mbox{\tiny$\begin{array}{c}\mu'\subset\mu^2\\
\mu'\subset\mu^S\\
\end{array}$}}C_{\mu^S,\mu'}\cdot X_{\iota_{\mu'}}
\end{equation}
where
\[
X_{\iota_{\mu'}}=(-1)^{\mathrm{rk}(G_{\iota_{\mu'}})+\mathrm{rk}(H_{\iota_{\mu'}})+\mathrm{rk}(T)+\mathrm{rk}(S)} \frac{|W_{T_1\times T_2,\theta\otimes1,\iota_{\mu'}}|}{|W_{G_{\iota_{\mu'}}}(T_1\times T_2)^F|},
\]
with $W_{T_1\times T_2,\theta\otimes1,\iota_{\mu'}}$ given by (\ref{wt}).
\end{proposition}
\begin{proof}
By  (\ref{x}), we have
\[
 \begin{aligned}
\langle R^{G}_{T_1\times T_2,\theta\otimes1},R_{S,1}^{H}\rangle _{H(\Fq)}
=\sum_{\mbox{\tiny$\begin{array}{c}\mu'\subset\mu\\
\mu'\subset\mu^S\\
\end{array}$}}\#[\iota_{\mu'}]\cdot X_{\iota_{\mu'}}.
\end{aligned}
\]
Note that $|W_{T_1\times T_2,\theta\otimes1,\iota_{\mu'}}|=0$ if $\mu'\not\subset\mu^2$.
\end{proof}

\section{Branching laws for $\UU_n(\Fq)$} \label{sec5}

In this section we study the branching of unipotent representations of finite unitary groups. We will prove the following result, which is the Bessel case of Theorem \ref{main}.

\begin{theorem}\label{th1}
Assume that $n > m$ and  $n-m$ is odd. Let $\lambda$ be a partition of $n$ into $k$ rows, and $\lambda'$ be the partition of $n-k$ obtained by removing the first column of $\lambda$. Let $\pi'$ be an irreducible representation of $\UU_m(\Fq)$. Then the following hold.

(i) If $m<n-k$, then
\[
m(\pi_\lambda, \pi') = 0.
\]

(ii) If $k$ is odd and $m=n-k$, then
\[
m(\pi_\lambda, \pi')=\left\{
\begin{array}{ll}
1, &  \textrm{if } \pi'\cong \pi_{\lambda'},\\
0, & \textrm{otherwise.}
\end{array}\right.
\]
\end{theorem}

\subsection{Basic case} In this subsection we first show that parabolic induction preserves multiplicities, and thereby make a reduction to the basic case. Then we explore Reeder's multiplicity formula in details in this case.
 Put
\begin{equation}
\GG_\ell:=\mathrm{Res}_{\mathbb{F}_{q^2}/\Fq}\GGL_\ell,
\end{equation}
so that $\GG_\ell(\Fq)=\mathrm{GL}_\ell(\mathbb{F}_{q^2})$. Let $P$ be an $F$-stable maximal parabolic subgroup of $\UU_{n+1}$ with Levi factor $\GG_\ell\times \UU_m$ (so that $n+1=m+2\ell$), and $\tau$ be an irreducible cuspidal representation of $\GGL_\ell(\mathbb{F}_{q^2})$. From \cite[Proposition 5.3]{GGP2}, we know that parabolic induction preserves multiplicities between cuspidal representations, namely,
\[
\langle \pi\otimes\bar{\nu}, \pi'\rangle _{H(\Fq)}=\langle I^{\UU_{n+1}}_{P}(\tau\otimes\pi'),\pi\rangle _{\UU_{n}(\Fq)}
\]
for cuspidal representations $\pi$ and $\pi'$ of $\UU_n(\Fq)$ and $\UU_m(\Fq)$ respectively. In the same manner, we have the following analog when $\pi$ is unipotent, which reduces the
calculation to the basic case.

\begin{proposition}\label{prop6.2}
Let $\pi$ be an irreducible unipotent representation of $\UU_n(\Fq)$, and $\pi'$ be an irreducible representation of $\UU_m(\Fq)$ with $n > m$ but $m \not\equiv n \ \mathrm{mod} \ 2$. Let $P$ be an $F$-stable maximal parabolic subgroup of $\UU_{n+1}$ with Levi factor $\GG_\ell \times \UU_m$ (so that $m + 2\ell = n + 1$), and $\tau$ be an irreducible cuspidal representation of $\GG_\ell\fq$ which is nontrivial if $\ell=1$. Then we have
\[
m(\pi, \pi')=\langle \pi\otimes \bar{\nu}, \pi'\rangle _{H(\Fq)}=\langle I^{\UU_{n+1}}_{P}(\tau\otimes\pi'),\pi\rangle _{\UU_n(\Fq)},
\]
where the data $(H,\nu)$ is given by (\ref{hnu}).
\end{proposition}

\begin{proof}
It can be proved in the same way as \cite[Theorem 15.1]{GGP1}. The cuspidality assumption of $\pi$ in \cite[Proposition 5.3]{GGP2} was used to obtain the following statement: for an $F$-stable maximal parabolic subgroup $P'$ of $\UU_n$ with Levi factor $\GG_\ell \times \UU_{m-1}$,
\[
\langle I^{\UU_n}_{P'}\left(\tau\otimes(\pi'|_{\UU_{m-1}(\Fq)})\right),\pi\rangle _{\UU_n(\Fq)}=0.
\]
Since in our case $\pi$ is unipotent, this  multiplicity is nonzero
only if $\tau$ and $\pi'|_{\UU_{m-1}(\Fq)}$ are both unipotent. It is well known that $\GG_\ell\fq$, $\ell>1$ has no unipotent cuspidal representations. By the assumption on $\tau$, it is not unipotent. Therefore the above multiplicity is zero.
The rest of the proof is the same as that of \cite[Theorem 15.1]{GGP1}.
\end{proof}

This result shows that  to calculate $\langle \pi_{\lambda}\otimes \bar{\nu} ,\pi'\rangle _{H(\Fq)}$, we only need to calculate $\langle I^{\UU_{n+1}}_{P} (\tau\otimes\pi'),\pm R^{\UU_n}_\lambda\rangle _{\UU_{n}(\Fq)}$ (cf. Theorem \ref{thm3.3}).
Since every representation of $\UU_{n+1}$ is a sum of Deligne-Lusztig characters, we now compute
\begin{equation}\label{equ6.1}
\langle R^{\UU_{n+1}}_{T_1\times T_2,\theta\otimes1},R^{\UU_n}_{\lambda}\rangle _{\UU_{n}(\Fq)}=\frac{1}{|S_n|}\sum_{w\in S_n}\sigma_\lambda(ww_0)\langle R^{\UU_{n+1}}_{T_1\times T_2,\theta\otimes1}, R_{T_w,1}^{\UU_n}\rangle _{\UU_{n}(\Fq)},
\end{equation}
 where $T_1$ is an $F$-stable minisotropic maximal torus of $\GG_\ell$, $T_2$ is an $F$-stable maximal torus of $\UU_{m}$, and $\theta$ is a regular character of $T_1^F$.

\begin{proposition}\label{prop6.5}
Let $T=T_1\times T_2$ be an $F$-stable maximal torus of $\UU_{n+1}$ corresponding to a partition $\mu=[\mu^1,\mu^2]$ of $n+1$.  Assume that
$1 \notin (T_1,\theta)$ and $\lambda$ is a partition of $n$. Then
\[
\begin{aligned}
\langle R^{\UU_{n+1}}_{T_1\times T_2,\theta\otimes1},R^{\UU_n}_\lambda\rangle _{\UU_{n}(\Fq)}=&\frac{(-1)^{\mathrm{rk}(T)}}{|S_n|}
\sum_{\mu'\subset\mu^2}(-1)^{n-|\mu'|}\frac{|W_{T_1\times T_2,\theta\otimes1,\iota_{\mu'}}|}{|W_{G_{\iota_{\mu'}}}(T_1\times T_2)^F|}
\\
&\times\sum_{\mu^*}\sum_{\mbox{\tiny$\begin{array}{c}w\in S_n\\
w\sim [\mu',\mu^*]\end{array}$}}
\sigma_\lambda(w)\rm{sgn}(w) C_{[\mu',\mu^*],\mu'}
\end{aligned}
\]
where $\mu^*$ runs over the partitions of $n-|\mu'|$ and $w\sim [\mu',\mu^*]$ means that $w$ corresponds to the partition $[\mu',\mu^*]$.
\end{proposition}

This proposition gives a combinatorial formula in terms of  the binomial coefficients
 (\ref{CC}), various partitions and Weyl group orders, which roughly asserts that to evaluate the multiplicity $\langle R^{\UU_{n+1}}_{T,\theta_T},R_{\lambda}^{\UU_n}\rangle _{\UU_n(\Fq)}$, one may separate the unipotent and non-unipotent parts in $\theta_T$. It is a general result and  in particular we do not assume that $T=T_1\times T_2$ is contained in an $F$-stable parabolic subgroup. We will see in Proposition \ref{prop6.8} below that for most $\mu^2$ satisfying $|\mu^2|\le n-k$, where $k$ is the number of rows in $\lambda$, one has
 \[
 \sum_{\mu^*}\sum_{\mbox{\tiny$\begin{array}{c}w\in S_n\\
w\sim [\mu',\mu^*]\end{array}$}}
\sigma_\lambda(w)\rm{sgn}(w) C_{[\mu',\mu^*],\mu'}=0.
 \]
Hence in these cases from Proposition \ref{prop6.5} it follows that
\[
\langle R^{\UU_{n+1}}_{T_1\times T_2,\theta\otimes1},R^{\UU_n}_\lambda\rangle _{\UU_{n}(\Fq)}=0.
\]
The combinatorial computation using the result in the nonvanishing case, combined with Proposition \ref{prop6.2},  will eventually yield exactly the expected multiplicity in Proposition \ref{prop6.9}.

\begin{proof}Let $\mu^w$ denote the partition of $n$ corresponding to $T_w$. By Proposition \ref{p1}, we have
\[
\begin{aligned}
&\langle R^{\UU_{n+1}}_{T_1\times T_2,\theta\otimes1},R^{\UU_n}_{\lambda}\rangle _{\UU_{n}(\Fq)}\\
=&\frac{1}{|S_n|}\sum_{w\in S_n}\sigma_\lambda(ww_0)\langle R^{\UU_{n+1}}_{T_1\times T_2,\theta\otimes1}, R_{T_w,1}^{\UU_n}\rangle _{\UU_{n}(\Fq)}\\
=&\frac{1}{|S_n|}\sum_{w\in S_n}\sigma_\lambda(ww_0)\sum_{\mbox{\tiny$\begin{array}{c}\mu'\subset\mu^2\\
\mu'\subset\mu^w\\
\end{array}$}}(-1)^{\mathrm{rk}(G_{\iota_{\mu'}})+\mathrm{rk}(H_{\iota_{\mu'}})+\mathrm{rk}(T)+\mathrm{rk}(T_w)}C_{\mu^w,\mu'} \frac{|W_{T_1\times T_2,\theta\otimes1,\iota_{\mu'}}|}{|W_{G_{\iota_{\mu'}}}(T_1\times T_2)^F|}\\
=&\frac{1}{|S_n|}
\sum_{\mu'\subset\mu^2}
\sum_{\mbox{\tiny$\begin{array}{c}w\in S_n\\
\mu^w\supset\mu'\\
\end{array}$}}
\sigma_\lambda(ww_0)(-1)^{\mathrm{rk}(G_{\iota_{\mu'}})+\mathrm{rk}(H_{\iota_{\mu'}})+\mathrm{rk}(T)+\mathrm{rk}(T_w)} C_{\mu^w,\mu'}
\frac{|W_{T_1\times T_2,\theta\otimes1,\iota_{\mu'}}|}{|W_{G_{\iota_{\mu'}}}(T_1\times T_2)^F|}\\
=&\frac{(-1)^{\mathrm{rk}(T)}}{|S_n|}
\sum_{\mu'\subset\mu^2}(-1)^{\mathrm{rk}(G_{\iota_{\mu'}})+\mathrm{rk}(H_{\iota_{\mu'}})}\frac{|W_{T_1\times T_2,\theta\otimes1,\iota_{\mu'}}|}{|W_{G_{\iota_{\mu'}}}(T_1\times T_2)^F|}
\sum_{\mbox{\tiny$\begin{array}{c}w\in S_n\\
\mu^w\supset\mu'\\
\end{array}$}}
\sigma_\lambda(ww_0)(-1)^{\mathrm{rk}(T_w)} C_{\mu^w,\mu'}\\
=&\frac{(-1)^{\mathrm{rk}(T)}}{|S_n|}
\sum_{\mu'\subset\mu^2}(-1)^{\mathrm{rk}(G_{\iota_{\mu'}})+\mathrm{rk}(H_{\iota_{\mu'}})}\frac{|W_{T_1\times T_2,\theta\otimes1,\iota_{\mu'}}|}{|W_{G_{\iota_{\mu'}}}(T_1\times T_2)^F|}
\sum_{\mbox{\tiny$\begin{array}{c}w\in S_n\\
\mu^w\supset\mu'\\
\end{array}$}}
\sigma_\lambda(ww_0)\rm{sgn}(ww_0) C_{\mu^w,\mu'}.
\end{aligned}
\]

By Lemma \ref{6.1}, (\ref{g}) and using the fact that
 \[
\mathrm{rk}(\UU_n)-\mathrm{rk}(\UU_{n-1})=\left\{\begin{array}{ll} 0, & \textrm{if }n\textrm{ is odd},\\
1, & \textrm{if }n\textrm{ is even},
\end{array}\right.
\]
we deduce that
\[
(-1)^{\mathrm{rk}(G_{\iota_{\mu'}})+\mathrm{rk}(H_{\iota_{\mu'}})}=(-1)^{n-|\mu'|}
\]
Then
\[
\begin{aligned}
&\frac{(-1)^{\mathrm{rk}(T)}}{|S_n|}
\sum_{\mu'\subset\mu^2}(-1)^{\mathrm{rk}(G_{\iota_{\mu'}})+\mathrm{rk}(H_{\iota_{\mu'}})}\frac{|W_{T_1\times T_2,\theta\otimes1,\iota_{\mu'}}|}{|W_{G_{\iota_{\mu'}}}(T_1\times T_2)^F|}
\sum_{\mbox{\tiny$\begin{array}{c}w\in S_n\\
\mu^w\supset\mu'\\
\end{array}$}}
\sigma_\lambda(ww_0)\rm{sgn}(ww_0) C_{\mu^w,\mu'}\\
=&\frac{(-1)^{\mathrm{rk}(T)}}{|S_n|}
\sum_{\mu'\subset\mu^2}(-1)^{n-|\mu'|}\frac{|W_{T_1\times T_2,\theta\otimes1,\iota_{\mu'}}|}{|W_{G_{\iota_{\mu'}}}(T_1\times T_2)^F|}
\sum_{\mbox{\tiny$\begin{array}{c}w\in S_n\\
\mu^w\supset\mu'\\
\end{array}$}}
\sigma_\lambda(ww_0)\rm{sgn}(ww_0) C_{\mu^w,\mu'}\\
=&\frac{(-1)^{\mathrm{rk}(T)}}{|S_n|}
\sum_{\mu'\subset\mu^2}(-1)^{n-|\mu'|}\frac{|W_{T_1\times T_2,\theta\otimes1,\iota_{\mu'}}|}{|W_{G_{\iota_{\mu'}}}(T_1\times T_2)^F|}
\sum_{\mu^*}\sum_{\mbox{\tiny$\begin{array}{c}w\in S_n\\
w\sim [\mu',\mu^*]\end{array}$}}
\sigma_\lambda(w)\rm{sgn}(w) C_{[\mu',\mu^*],\mu'}
\end{aligned}
\]
where $\mu^*$ runs over the partitions of $n-|\mu'|$, and $w\sim [\mu',\mu^*]$ means that $w$ corresponds to the partition $[\mu',\mu^*]$.
\end{proof}

\subsection{Proof of Theorem \ref{th1}}
We keep the notations $T_1$, $T_2$ and $\mu=[\mu^1,\mu^2]$ in Proposition \ref{prop6.5}. For a fixed $\mu'\subset\mu^2$, let $T_{\mu'}$ be the $F$-stable maximal torus of $\UU_{|\mu'|}$ corresponding to the partition $\mu'$. To ease notations,  below we write $W_n(T)$ for $W_{\UU_n}(T)$. Then
\begin{equation}
|W_{T_1\times T_2,\theta\otimes1,\iota_{\mu'}}|=C_{\mu^2,\mu'}|W_{|\mu'|}(T_{\mu'})^F|\cdot
|W_{G_{\iota_{\mu'}}}(T_1\times T_2)^F|.
\end{equation}
In particular, if $\mu'=\mu^2$, then
\begin{equation}\label{eq-5.4}
|W_{T_1\times T_2,\theta\otimes1,\iota_{\mu^2}}|=|W_{|\mu^2|}(T_2)^F|\cdot
|W_{n+1-|\mu^2|}(T_1)^F|.
\end{equation}

By the Littlewood-Richardson rule, we have the following result.

\begin{lemma}\label{lem5.6}
Let $\lambda=[\lambda_1,\ldots,\lambda_\ell]$ be a partition of $n$. If $m>\lambda_1$, then
\[
\langle\sigma_{\lambda},1\rangle_{S_m}=0.
\]
Let $\lambda'=[\lambda_2,\ldots, \lambda_\ell]$ be the partition of $n-\lambda_1$ obtained by
removing the first row of $\lambda$, and $\sigma$ be an irreducible representation of $S_{n-\lambda_1}$. Then
\[
\langle\sigma_{\lambda},\sigma\otimes1\rangle_{S_{n-\lambda_1}\times S_{\lambda_1}}=\left\{ \begin{array}{ll}1, & \textrm{if }\sigma=\sigma_{\lambda'}, \\
0, & \textrm{otherwise.}\end{array}\right.
\]
In particular,
\[
\langle\sigma_{\lambda},1\rangle_{S_{\lambda_1}}=1.
\]
\end{lemma}

\begin{proof}
It can be proved in the same way as \cite[Proposition 3.4]{AM}. It should be mentioned that our notation of the representation of $\sigma_{\lambda}$ differs from that of \cite{AM}, where the representation  $\sigma_{\lambda}$ in \cite{AM} is equal to $\sigma_{{}^t\lambda}$ in our  paper.
\end{proof}

Now we can explicitly calculate the multiplicity formula for $\langle R^{\UU_{n+1}}_{T_1\times T_2,\theta\otimes1},R^{\UU_n}_\lambda\rangle _{\UU_{n}(\Fq)}$.

\begin{proposition} \label{prop6.8} Keep the notations and assumptions in Proposition \ref{prop6.5}.
Let $\lambda$ be a partition of $n$ into $k$ rows, and $\lambda'$ be the partition of $n-k$ obtained by removing the first column of $\lambda$. Then the following hold.

(i) If $|\mu'|<n-k$, then
 \[
 \sum_{\mu^*}\sum_{\mbox{\tiny$\begin{array}{c}w\in S_n\\
w\sim [\mu',\mu^*]\end{array}$}}
\sigma_\lambda(w)\rm{sgn}(w) C_{[\mu',\mu^*],\mu'}=0.
 \]
(ii) If $|\mu^2|=n-k$,  then
\[
\langle R^{\UU_{n+1}}_{T_1\times T_2,\theta\otimes1},R^{\UU_n}_\lambda\rangle _{\UU_{n}(\Fq)}=(-1)^{\mathrm{rk}(T)+k}\sigma_{{}^t\lambda'}(w_{\mu^2}),
\]
where $w_{\mu^2}$ is an element in $S_m$  corresponding to the partition $\mu^2$.
\end{proposition}

\begin{proof}
Denote by  $[w]$ the conjugacy class of a element $w$ of $S_n$. We have that
\[
\begin{aligned}
& \sum_{\mu^*}\sum_{\mbox{\tiny$\begin{array}{c}w\in S_n\\
w\sim [\mu',\mu^*]\end{array}$}}
\sigma_\lambda(w)\rm{sgn}(w) C_{[\mu',\mu^*],\mu'}\\
=& \sum_{\mu^*}
\#[w_{[\mu',\mu^*]}]C_{[\mu',\mu^*],\mu'}\sigma_{^t\lambda}(w_{[\mu',\mu^*]})\\
=& \sum_{\mu^*}
C_{[\mu',\mu^*],\mu'}\cdot\frac{|S_n|}{|W_n(T_{[\mu',\mu^*]})^F|}\sigma_{{}^t\lambda}(w_{[\mu',\mu^*]})\\
=& \sum_{\mu^*}
C_{[\mu',\mu^*],\mu'}\cdot\frac{|S_n|}{C_{[\mu',\mu^*],\mu'}\cdot|W_{|\mu'|}(T_{\mu'})^F|\cdot|W_{|\mu^*|}(T_{\mu^*})^F|}\sigma_{{}^t\lambda}(w_{[\mu',\mu^*]})\\
=&\frac{|S_n|}{|W_{|\mu'|}(T_{\mu'})^F|\cdot|S_{n-|\mu'|}|}
\sum_{\mu^*}
\frac{|S_{n-|\mu'|}|}{|W_{|\mu^*|}(T_{\mu^*})^F|}\sigma_{{}^t\lambda}(w_{[\mu',\mu^*]})\\
=&\frac{|S_n|}{|W_{|\mu'|}(T_{\mu'})^F|\cdot|S_{n-|\mu'|}|}
\sum_{\mu^*}
\frac{|S_{n-|\mu'|}|}{|W_{|\mu^*|}(T_{\mu^*})^F|}\sigma_{{}^t\lambda}((w_{\mu'},w_{\mu^*}))\\
=&\frac{|S_n|}{|W_{|\mu'|}(T_{\mu'})^F|\cdot|S_{n-|\mu'|}|}
\sum_{\mbox{\tiny$\begin{array}{c}w\in S_{n-|\mu'|}\end{array}$}}
\sigma_{{}^t\lambda}((w_{\mu'},w)),\\
\end{aligned}
\]
noting that $(w_{\mu'},w_{\mu^*})\in S_{|\mu'|}\times S_{n-|\mu'|}\subset S_n$ in the second last equation above.
Applying Lemma \ref{lem5.6}, one deduces that
\[
\begin{aligned}
\sum_{\mbox{\tiny$\begin{array}{c}w\in S_{n-|\mu'|}\end{array}$}}
\sigma_{{}^t\lambda}((w_{\mu'},w))=
\left\{ \begin{array}{ll}
|S_{n-|\mu'|}|\sigma_{{}^t\lambda'}(w_{\mu'}), & \textrm{if }|\mu'|=n-k, \\
0, & \textrm{if }|\mu'|<n-k, \end{array}\right.
\end{aligned}
\]
where $\lambda'$ is the partition of  $n-k$ obtained by removing the first column of $\lambda$.

If $|\mu^2|=n-k$,  then $|\mu'|\le n-k$ with equality hold only if $\mu'=\mu^2$. Applying (\ref{eq-5.4}) and part (i) of the proposition, we have that
\[
\begin{aligned}
&\langle R^{\UU_{n+1}}_{T_1\times T_2,\theta\otimes1},R^{\UU_n}_\lambda\rangle _{\UU_{n}(\Fq)}\\
=&\frac{(-1)^{\mathrm{rk}(T)+k}}{|S_n|}
\frac{|W_{T_1\times T_2,\theta\otimes1,\iota_{\mu^2,\mu^2}}|}{|W_{G_{\iota_{\mu^2}}}(T_1\times T_2)^F|}\sum_{\mu^*}\sum_{\mbox{\tiny$\begin{array}{c}w\in S_n\\
w\sim [\mu^2,\mu^*]\end{array}$}} C_{[\mu^2,\mu^*],\mu^2}\mathrm{sgn}(w)\sigma_{\lambda}(w)\\
=&\frac{(-1)^{\mathrm{rk}(T)+k}}{|S_n|}
\frac{|W_{T_1\times T_2,\theta\otimes1,\iota_{\mu^2,\mu^2}}|}{|W_{G_{\iota_{\mu^2}}}(T_1\times T_2)^F|}
\cdot\frac{|S_n|\cdot|S_{n-|\mu^2|}|}{|W_{|\mu^2|}(T_{\mu^2})^F|\cdot|S_{n-|\mu^2|}|}\sigma_{{}^t\lambda'}(w_{\mu^2})\\
=&\frac{(-1)^{\mathrm{rk}(T)+k}}{|S_n|}
\frac{|W_{n+1-|\mu^2|}(T_1)^F|\cdot|W_{|\mu^2|}(T_2)^F|}{ |W_{n+1-|\mu^2|}(T_1)^F| }
\cdot\frac{|S_n|\cdot|S_{n-|\mu^2|}|}{|W_{|\mu^2|}(T_{\mu^2})^F|\cdot|S_{n-|\mu^2|}|}\sigma_{{}^t\lambda'}(w_{\mu^2})\\
=&(-1)^{\mathrm{rk}(T)+k}\sigma_{{}^t\lambda'}(w_{\mu^2}),\\
\end{aligned}
\]
which gives (ii).
\end{proof}

Combining the previous results, we are  ready to give the main result of this subsection.

\begin{proposition}\label{prop6.9}  Keep the notations and assumptions in Proposition \ref{prop6.5}. Let $\lambda$ be a partition of $n$ into $k$ rows, and $\lambda'$ be the partition of $n-k$ obtained by removing the first column of $\lambda$.  Then the following hold.

(i) If $m<n-k$, then
\[
\langle R^{\UU_{n+1}}_{T_1\times T_2,\theta\otimes1},R^{\UU_n}_\lambda\rangle _{\UU_{n}(\Fq)}=0.
\]
In particular, in this case if $P$ is an $F$-stable maximal parabolic subgroup of $\UU_{n+1}$ with Levi factor $\GG_\ell\times \UU_m$, and $\tau$ is an irreducible cuspidal representation of
$\GGL_\ell(\mathbb{F}_{q^2})$ which is notrivial if $\ell=1$, then for any representation $\pi'$ of $\UU_m(\Fq)$,
\[
\langle I^{\UU_{n+1}}_{P}(\tau\otimes\pi'),R^{\UU_n}_\lambda\rangle _{\UU_{n}(\Fq)}=0.
\]

(ii) If $k=2\ell-1$ so that $m=n-k$, then for an irreducible representation $\pi$ of $\UU_m(\Fq)$,
\[
\langle I^{\UU_{n+1}}_{P}(\tau\otimes\pi'), \pi_\lambda\rangle_{\UU_n(\Fq)}=\left\{
\begin{array}{ll}
1, &  \textrm{if }\pi'\cong\pi_{\lambda'},\\
0 & \textrm{otherwise.}
\end{array}
\right.
\]
\end{proposition}

\begin{proof} (i) follows from Proposition \ref{prop6.8}. For (ii), if $\pi'=\pi_{\nu}$ is unipotent with a partition $\nu$ of $m$, then by Proposition \ref{prop6.8} (ii) we obtain that
\[
\begin{aligned}
& \langle I^{\UU_{n+1}}_{P}(\tau\otimes\pi'),R^{\UU_n}_\lambda\rangle _{\UU_{n}(\Fq)}\\
=&\pm\langle I^{\UU_{n+1}}_{P}(\tau\otimes R^{\UU_m}_{\nu}),R^{\UU_n}_\lambda\rangle _{\UU_{n}(\Fq)}\\
=&\pm  \frac{1}{|S_m|}\sum_{w\in S_m}\sigma_{\nu}(ww_0)\langle R^{\UU_{n+1}}_{T_1\times T_w,\theta\otimes1},R^{\UU_n}_\lambda\rangle _{\UU_{n}(\Fq)}\\
=&\pm \frac{1}{|S_m|}\sum_{\mu^2}\sum_{\mbox{\tiny$\begin{array}{c}w\in S_m\\
w\sim \mu^2\end{array}$}}
\sigma_{\nu}(w)\langle R^{\UU_{n+1}}_{T_1\times T_{ww_0},\theta\otimes1},R^{\UU_n}_\lambda\rangle _{\UU_{n}(\Fq)}\\
=&\pm \frac{1}{|S_m|}\sum_{\mu^2}\#[w_{\mu^2}]
\sigma_{\nu}(w_{\mu^2})\langle R^{\UU_{n+1}}_{T_1\times T_{\mu^2},\theta\otimes1},R^{\UU_n}_\lambda\rangle _{\UU_{n}(\Fq)}\\
=&\pm \frac{1}{|S_m|}\sum_{\mu^2}\#[w_{\mu^2}]
\sigma_{\nu}(w_{\mu^2})\rm{sgn}(w_{\mu^2})\sigma_{^t\lambda'}(w_{\mu^2})\\
=&\pm\langle\sigma_{^t\nu},\sigma_{^t\lambda'}\rangle_{S_m}\\
=& \left\{
\begin{array}{ll}
\pm 1, &  \textrm{if }\nu=\lambda',\\
0 & \textrm{otherwise,}
\end{array}
\right.
\end{aligned}
\]
where $\mu^2$ runs over partitions of $m$.

Suppose  that $\pi'$ is not unipotent. Then there is a geometric conjugacy class $\kappa$ for $\UU_m$ such that $\pi'\in R_{T,\theta}^G$ only if the pair $(T,\theta)\in \kappa$. Since $\pi'$ is not unipotent, there is an integer $m'<m$ such that every pair $(T,\theta)\in \kappa$ has the form $(T,\theta)=(T_1'\times T_2',\theta'\otimes 1)$ with $1\notin (T_1' ,\theta')$ where $T_1'$ is an $F$-stable maximal torus of $\UU_{m'}$, and $T_2$ is an $F$-stable maximal torus of $\UU_{n+1-m'}$. Hence for any pair $(T,\theta)\in \kappa$,
\[
\langle R^{\UU_{n+1}}_{T_1\times(T_1'\times T_2'),\theta\otimes(\theta'\otimes 1)},R^{\UU_n}_\lambda\rangle _{\UU_{n}(\Fq)}=\langle R^{\UU_{n+1}}_{(T_1\times T_1')\times T_2',(\theta\otimes\theta')\otimes 1},R^{\UU_n}_\lambda\rangle _{\UU_{n}(\Fq)}=0
\]
noting that $1\notin ((T_1\times T_1'),(\theta\otimes\theta'))$ and $T_1\times T_1'$ an $F$-stable maximal torus of $\UU_{k+m'}$. Since $\pi'$ is a linear combination of $R^{\UU_{n+1}}_{T_1\times(T_1'\times T_2'),\theta\otimes(\theta'\otimes 1)}$ for pairs $(T_1'\times T_2',\theta'\otimes 1)\in \kappa$, we obtain that
\[
\langle I^{\UU_{n+1}}_{P}(\tau\otimes\pi'), \pi_\lambda\rangle_{\UU_n(\Fq)}=0.
\]
This finishes the proof of (ii).
\end{proof}

Finally, Theorem \ref{th1} follows from Proposition \ref{prop6.2} and Proposition \ref{prop6.9}.

\section{Fourier-Jacobi case}\label{sec6}

We have established the Bessel descents of unipotent representations of finite unitary groups. In this section we deduce the Fourier-Jacobi case from the Bessel case by the standard arguments of theta correspondence and see-saw dual pairs, which are used in the proof of local Gan-Gross-Prasad conjecture (see \cite{GI, Ato}).

\subsection{Lusztig correspondence}

Let $G^*$ be the dual group of $G$. We still denote the Frobenius endomorphism of $G^*$ by $F$, and $G^{*F}$ the group of rational points. It is known that there is a bijection between the set of $G^F$-conjugacy classes of $(T, \theta)$ and the set of $G^{*F}$-conjugacy classes of $(T^*, s)$ where $T^*$ is a $F$-stable maximal torus in $G^*$ and $s \in   T^{*F}$ . If $(T, \theta)$ corresponds to $(T^*, s)$, then $R_{T,\theta}^G$ will be also denoted by $R_{T^*,s}^G$.
For a semisimple element $s \in G^{*F}$, define
\[
\mathcal{E}(G^F,s) = \{ \chi \in \mathcal{E}(G^F) | \langle \chi, R_{T^*,s}^G\rangle \ne 0\textrm{ for some }T^*\textrm{ containing }s \}.
\]
The set $\mathcal{E}(G^F,s)$ is called a Lusztig series, and it is known that $\mathcal{E}(G^F)$ is partitioned into
Lusztig series indexed by the conjugacy classes $(s)$ of semisimple elements $s$,
i.e.,
\[
\mathcal{E}(G^F)=\coprod_{(s)}\mathcal{E}(G^F,s).
\]

The following result is fundamental for the classification of $\mathcal{E}(G)$:

\begin{proposition}[Lusztig]\label{Lus}
There is a bijection
\[
\mathcal{L}_s:\mathcal{E}(G^F,s)\to \mathcal{E}(C_{G^{*F}}(s),1).
\]
extending by linearity to virtual characters satisfying the condition
\[
\mathcal{L}_s(\varepsilon_G R^G_{T^*,s})=\varepsilon_{C_{G^{*F}}(s)} R^{C_{G^{*F}}(s)}_{T^*,1}.
\]
\end{proposition}

For later use, we prove the following irreducibility result.

 \begin{lemma}\label{irr}
Let $\tau$ be an irreducible cuspidal representation of $\GG_\ell\fq$, which is not conjugate self-dual.  Then $R^{\UU_n}_{\GG_\ell \times \UU_{n-2\ell}}(\tau\otimes\pi_\lambda)$ is irreducible for any unipotent representation $\pi_\lambda$ of $\UU_{n-2\ell}\fq$.
\end{lemma}

\begin{proof}
Write $\tau=\pm R^{\GG_\ell}_{T^*,s}$ for some $F$-stable minisotrpic maximal torus $T$ of $\GG_\ell$ and a regular semisimple element $s\in T^{*F}$.  Since $\tau$ is not conjugate self-dual, $s$ is in fact regular in $\UU_{2\ell}\fq$. Then $C_{\GG_\ell}(s)\cong C_{\UU_{2\ell}}(s)\cong T^{*F}$, and
 \[
 \mathcal{L}_s(\tau)= \mathcal{L}_s(\pm R^{\GG_\ell}_{T^*,s})=\pm R^{T}_{T^*,1}
 \]
 is the trivial representation of $T^F$.
We have
\[
R^{\UU_n}_{\GG_\ell \times \UU_{n-2\ell}}(\tau\otimes\pi_\lambda)=\pm \frac{1}{|S_{n-2\ell}|}\sum_{w\in S_{n-2\ell}}\sigma_\lambda(ww_0)R_{T^*\times T_w^*,\theta\otimes1}^{\UU_n}.
\]
It follows that
\[
\mathcal{L}_{(s,1)}(R^{\UU_n}_{\GG_\ell \times \UU_{n-2\ell}}(\tau\otimes\pi_\lambda))=\pm\frac{1}{|S_{n-2\ell}|}\sum_{w\in S_{n-2\ell}}\sigma_\lambda(ww_0)R_{T^*_w,1}^{\UU_{n-2\ell}}\otimes R^{C_{\UU_{2\ell}}(s)}_{T^*,1}\cong\pi_\lambda\otimes\bf{1},
\]
which is irreducible as a representation of $\UU_{n-2\ell}\fq\times T^F$. Hence $I^{\UU_n}_{\GG_\ell \times \UU_{n-2\ell}}(\tau\otimes\pi_\lambda)$ is irreducible by Lusztig correspondence.
\end{proof}

\subsection{Weil representation and theta lifting}
Let $\omega_{\rm{Sp}_{2N}}$ be the character of the Weil representation (cf. \cite{Ger}) of the finite symplectic group $\rm{Sp}_{2N}(\Fq)$, which depends on a  nontrivial additive character $\psi$ of $F$.
Let $(G, G^{\prime})$ be a reductive dual pair in $\rm{Sp}_{2N}$, and write
$\omega_{G,G'}$ for the restriction of $\omega_{\rm{Sp}_{2N}}$ to $G^F\times G'^F$. Then it decomposes into a direct sum
\[
\omega_{G,G'}=\bigoplus_{\pi,\pi'} m_{\pi,\pi '}\pi\otimes\pi '
\]
where $\pi$ and $\pi '$ run over irreducible representations of $G^F$ and $G'^F$ respectively, and $m_{\pi,\pi'}$ are nonnegative integers.. We can rearrange this decomposition as
\[
\omega_{G,G'}=\bigoplus_{\pi} \pi\otimes\Theta_{G,G'}(\pi )
\]
 where $\Theta_{G, G'}(\pi ) = \bigoplus_{\pi'} m_{\pi,\pi '}\pi '$ is a (not necessarily irreducible) representation of $G'^F$, called the (big) theta lifting of $\pi$ from $G^F$ to $G'^F$. Write $\pi'\subset \Theta_{G'}(\pi)$ if $\pi\otimes\pi'$ occurs in $\omega_{G,G'}$, i.e. $m_{\pi, \pi'}\neq 0$. We remark that  even if $\Theta_{G,G'}(\pi)=:\pi'$ is irreducible, one only has
 \[
 \pi\subset \Theta_{G',G}(\pi'),
 \]
 while the equality does not  necessarily hold.

Consider a dual pair of unitary groups $(G,G')=(\UU_n, \UU_{n'})$ in $\rm{Sp}_{2nn'}$. Denote $\omega_{G, G'}$ by $\omega_{n, n'}$, and $\Theta_{G,G'}$ by $\Theta_{n,n'}$. In particular, denote by $\omega_n$ the restriction of $\omega_{\rm{Sp}_{2n}}$ to $\rm{U}_n(\Fq)$. By \cite[Theorem 3.5]{AM}, theta lifting between unitary groups sends unipotent representations to unipotent representations, and we will recall the explicit correspondence below.

We say that two partitions $\mu=[\mu_i]$ and $\mu'=[\mu_i']$ are {\sl close} if  $|\mu_i-\mu_i'|\le1$ for every $i$, and that $\mu$ is {\sl even} if $\#\{i|\mu_i=j\}$ is even for any $j>0$, i.e. every part of $\mu$ occurs with even multiplicities. Let
\[
\mu\cap\mu'=[\mu_i]_{\{ i| \mu_i=\mu_i'\}}
\]
be the partition formed by the common parts of $\mu$ and $\mu'$. Following \cite{AMR}, we say that $\mu$ and $\mu'$ are {\sl 2-transverse} if they are close and $\mu\cap \mu'$ is even.
In particular, if $\mu$ and $\mu'$ are close and $\mu\cap \mu'=\emptyset$, then $\mu$ and $\mu'$ are 2-transverse, and in this case we say that they are {\sl transverse}. For example, let
$\lambda=[\lambda_1,\ldots, \lambda_k]$ be a partition of $n$, and let $\lambda_*=[\lambda_2,\ldots, \lambda_k]$ be the partition of $n-\lambda_1$ obtained by removing the first row of $\lambda$. Then ${}^t\lambda$ and ${}^t\lambda_*$ are transverse. Moreover, $\lambda_*$ is the unique partition of $n-\lambda_1$ such that ${}^t\lambda$ and ${}^t\lambda_*$ are 2-transverse.

For  partitions $\lambda$ and $\lambda'$ of $n$ and $n'$ respectively, denote the multiplicity of $\pi_{\lambda}\otimes\pi_{\lambda'}$ in $\omega_{n,n'}$ by $m_{\lambda,\lambda'}$. By \cite{AMR} Theorem 4.3, Lemma 5.3 and Lemma 5.4, we have

\begin{proposition}\label{7.1}
With above notations,
\[
m_{\lambda,\lambda'}=\left\{
\begin{array}{ll}
1, & \textrm{if  }{}^t\lambda\textrm{ and }{}^t\lambda'\textrm{ are }2\textrm{-transverse},\\
0, & \textrm{otherwise.}
\end{array}
\right.
\]
In other words,
\[
\Theta_{n,n'}(\pi_{\lambda})=\bigoplus_{\mbox{\tiny$\begin{array}{c}{}^t\lambda\ \mathrm{and}\ {}^t\lambda'\mathrm{\ are\ }2\textrm{-}\mathrm{transverse}\\
|\lambda'|=n' \end{array}$}}\pi_{\lambda'}
\]
In particular, $\Theta_{n,n'}(\pi_{\lambda})=0$ if $n'<n-\lambda_1$.
\end{proposition}

The next result shows that theta lifting and parabolic induction are compatible.

\begin{proposition}\label{7.2}
Let $\tau$ be an irreducible cuspidal representation of $\GG_\ell\fq$ which is nontrivial if $\ell=1$, $\pi$ be an irreducible representation of $\UU_n(\Fq)$, and  $\pi':=\Theta_{n,n'}(\pi)$. For any irreducible $\rho\subset R^{\UU_{n+2\ell}}_{\GG_\ell\times \UU_{n}}(\tau \otimes \pi)$ and $\rho'\subset \Theta_{n+2\ell, n'+2\ell}(\rho)$,  we have
\[
\rho'\subset R^{\UU_{n'+2\ell}}_{\GG_\ell\times \UU_{n'}}(\tau\otimes \pi').
\]
In particular, if $R^{\UU_{n+2\ell}}_{\GG_\ell\times \UU_{n}}(\tau \otimes \pi)$ is irreducible, then
\[
\Theta_{n+2\ell, n'+2\ell} (R^{\UU_{n+2\ell}}_{\GG_\ell\times \UU_{n}}(\tau \otimes \pi))= R^{\UU_{n'+2\ell}}_{\GG_\ell\times \UU_{n'}}(\tau\otimes \pi').
\]
\end{proposition}

\begin{proof}
Let ${}^*R^{\rm{U}_{n+2\ell}}_{\GG_\ell\times \UU_{n}}$ be the Jacquet functor, which is adjoint to the parabolic induction $R^{\rm{U}_{n+2\ell}}_{\GG_\ell\times \UU_{n}}$. By \cite[Chap. 3, IV, th. 5]{MVW},
\[
({}^*R^{\rm{U}_{n+2\ell}}_{\GG_\ell\times \UU_n}\otimes 1)(\omega_{n+2\ell, n'+2\ell})=
\bigoplus^{\ell}_{i=0}R^{\UU_{n}\times \GG_\ell \times \UU_{n'+2\ell}}_{\UU_n\times (\GG_{\ell-i}\times \GG_i)\times \GG_i\times \UU_{n'+2(\ell-i)}}(\omega_{n, n'+2(\ell-i)}\otimes 1_{\GG_{\ell-i}(\Fq)}\otimes R^{\GG_i}),
\]
where $R^{\GG_i}$ is the regular representation of $\GG_i(\Fq)=\GGL_i(\mathbb{F}_{q^2})$. Since $\tau$ is cuspidal, $\tau\otimes\pi$ only appears in the summand with $i=\ell$, and the proposition follows easily.
\end{proof}

\subsection{See-saw dual pairs}
Recall the general formalism of see-saw dual pairs. Let $(G, G')$ and $(H, H')$ be two reductive dual pairs in a symplectic group $\rm{Sp}(W)$ such that $H \subset G$ and $G' \subset H'$.  Then there is a see-saw diagram
\[
\setlength{\unitlength}{0.8cm}
\begin{picture}(20,5)
\thicklines
\put(6.8,4){$G$}
\put(6.8,1){$H$}
\put(12.1,4){$H'$}
\put(12,1){$G'$}
\put(7,1.5){\line(0,1){2.1}}
\put(12.3,1.5){\line(0,1){2.1}}
\put(7.5,1.5){\line(2,1){4.2}}
\put(7.5,3.7){\line(2,-1){4.2}}
\end{picture}
\]
and the associated see-saw identity
\[
\langle \Theta_{G',G}(\pi_{G'}),\pi_H\rangle_H =\langle \pi_{G'},\Theta_{H,H'}(\pi_H)\rangle_{G'},
\]
where $\pi_H$ and $\pi_{G'}$ are representations of $H$ and $G'$ respectively.

In our case, if we put
\[
G=\UU_n\times \UU_n,\quad  G'=\UU_n\times \UU_1, \quad H=\UU_n,\ \mathrm{and}\ H'=\UU_{n+1},
\]
then the left-hand side of the see-saw identity concerns the basic case of  Fourier-Jacobi model whereas the right-hand side
concerns the basic case of Bessel model. In general, we need Proposition \ref{prop6.2} and the following result.
\begin{proposition}\label{7.3}
Let $\pi$ be an irreducible unipotent representation of $\UU_n(\Fq)$, and $\pi'$ be an irreducible representation of $\UU_m(\Fq)$ with $n > m$ and $m \equiv n \ \mathrm{mod}  \ 2$. Let $P$ be an $F$-stable maximal parabolic subgroup of $\UU_m$ with Levi factor $\GG_\ell \times \UU_m$ (so that $m + 2\ell = n $) and $\tau$ be an irreducible cuspidal representation of $\GG_\ell\fq$ which is nontrivial if $\ell=1$. Then we have
\[
m(\pi, \pi')=\langle \pi\otimes\bar{\nu},\pi'\rangle_{H(\Fq)}=\langle  \pi\otimes \omega_n, I_{P}^{\UU_{n}}(\tau\otimes\pi')\rangle _{\UU_{n}(\Fq)},
\]
where the data $(H,\nu)$ is given by (\ref{hnu'}).
\end{proposition}

Similar to Proposition \ref{prop6.2}, the proof of Proposition \ref{7.3} is an adaptation of that of \cite[Theorem 16.1]{GGP1} to unipotent representations, which will be omitted here.
As mentioned in the Introduction, our formulation of multiplicities differs  from that in the Gan-Gross-Prasad conjecture  by taking contragradient of $\pi$, but we may take the advantage that unipotent representations are self-dual. 

\begin{proposition}\label{7.4}
If $\pi$ is a unipotent representation of $\UU_n(\Fq)$, then
$
\pi\cong \pi^\vee.
$
\end{proposition}

\begin{proof}
By Theorem \ref{thm3.3}, it suffices to show that every $R^{\UU_n}_{T,1}$ is self-dual.
By (\ref{dl}), for $y=su\in\UU_n\fq$, we have 
\[
R^{\UU_n}_{T,1}(y)=\frac{1}{|C^{0}(s)^F|}\sum_{g\in G^F,s^g \in T^F}Q^{C^{0}(s)}_{{}^{g}T}(u).
\]
It is well-known that $C^{0}(s)$ is a product of general linear groups and unitary groups,  
hence $u$ and $u^{-1}$ are conjugate in $C^0(s)^F$, which gives that
\[
Q^{C^{0}(s)}_{{}^{g}T}(u)=Q^{C^{0}(s)}_{{}^{g}T}(u^{-1}).
\]
It follows that
\[
\begin{aligned}
R^{\UU_n}_{T,1}(y^{-1})=&\frac{1}{|C^{0}(s^{-1})^F|}\sum_{g\in G^F,(s^{-1})^g \in T^F}Q^{C^{0}(s^{-1})}_{{}^{g}T}(u^{-1})\\
=&\frac{1}{|C^{0}(s)^F|}\sum_{g\in G^F,s^g \in T^F}Q^{C^{0}(s)}_{{}^{g}T}(u)\\
=&R^{\UU_n}_{T,1}(y),
\end{aligned}
\]
which proves that $R^{\UU_n}_{T,1}$ is self-dual.
\end{proof}

Finally we are ready to prove the Fourier-Jacobi case of Theorem \ref{main}.

\begin{theorem}\label{th2}
Assume that $n > m$ and  $n-m$ is even. Let $\lambda$ be a partition of $n$ into $k$ rows, and $\lambda'$ be the partition of $n-k$ obtained by removing the first column of $\lambda$. Let $\pi'$ be an irreducible representation of $\UU_m(\Fq)$. Then the following hold.

(i) If $m<n-k$, then
\[
m(\pi_\lambda, \pi') = 0.
\]

(ii) If $k$ is even and $m=n-k$, then
\[
m(\pi_\lambda, \pi')=\left\{
\begin{array}{ll}
1, &  \textrm{if } \pi'\cong \pi_{\lambda'},\\
0, & \textrm{otherwise.}
\end{array}\right.
\]
\end{theorem}

\begin{proof} Let $\tau$ be an irreducible cuspidal representation of $\GG_\ell\fq$ with $\ell=(n-m)/2$, which is not conjugate self-dual.
By Proposition \ref{7.3}, we only need to compute
\[
\langle \pi_\lambda\otimes \omega_n, I^{\UU_{n}}(\tau\otimes\pi')\rangle _{\UU_{n}(\Fq)}.
\]
Here and in the sequel, we drop the subcripts of various $F$-stable parabolic subgroups from the induced representations, which should be clear from the context.

(i) Following the proof of Proposition \ref{7.4}, write $\lambda=[\lambda_1,\ldots, \lambda_k]$ and put $\lambda_*=[\lambda_2,\ldots, \lambda_k]$, so that  we have $\Theta_{n,n-\lambda_1}(\pi_{\lambda})=\pi_{\lambda_*}$.
Consider the see-saw diagram
\[
\setlength{\unitlength}{0.8cm}
\begin{picture}(20,5)
\thicklines
\put(6.8,4){$\UU_n\times \UU_n$}
\put(7.5,1){$\UU_n$}
\put(12.7,4){$\UU_{n-\lambda_1+1}$}
\put(12,1){$\UU_{n-\lambda_1}\times \UU_1$}
\put(7.7,1.5){\line(0,1){2.1}}
\put(12.8,1.5){\line(0,1){2.1}}
\put(8,1.5){\line(2,1){4.2}}
\put(8,3.7){\line(2,-1){4.2}}
\end{picture}
\]
By Proposition \ref{7.1},
\[
\langle \pi_\lambda\otimes \omega_n, I^{\UU_{n}}(\tau\otimes\pi')\rangle_{\UU_{n}(\Fq)}\\
\leq  \langle \Theta_{ n-\lambda_1,n}(\pi_{\lambda_*})\otimes\omega_n, I^{\UU_{n}}(\tau\otimes\pi')\rangle _{\UU_{n}(\Fq)}.
\]
For any irreducible $\rho\subset I^{\UU_{n}}(\tau\otimes\pi')$, by the see-saw identity one has
\[
\langle \Theta_{ n-\lambda_1,n}(\pi_{\lambda_*})\otimes\omega_n, \rho\rangle _{\UU_{n}(\Fq)}= \langle \pi_{\lambda_*}, \Theta_{n, n-\lambda_1+1}(\rho) \rangle _{\UU_{n-\lambda_1}(\Fq)}
\]
By Proposition \ref{7.2}, for any irreducible $\rho'\subset \Theta_{n, n-\lambda_1+1}(\rho)$, one has
\begin{align*}
 \langle \pi_{\lambda_*}, \rho' \rangle _{\UU_{n-\lambda_1}(\Fq)}
&\leq  \langle \pi_{\lambda_*}, I^{\UU_{n-\lambda_1+1}}(\tau\otimes\Theta_{m, m-\lambda_1+1}(\pi')) \rangle _{\UU_{n-\lambda_1}(\Fq)} \\
&= m(\pi_{\lambda_*}, \Theta_{m, m-\lambda_1+1}(\pi')).
\end{align*}
If $m<n-k$, then by Theorem \ref{th1} (i) the above multiplicity vanishes, noting that
\[
m-\lambda_1+1<n-\lambda_1-(k-1).
\]

(ii) Assume that $k$ is even and $m=n-k$. If $\pi'$ is not unipotent, then $\Theta_{m, m-\lambda_1+1}(\pi')$ has no unipotent components, hence by Theorem \ref{th1} (ii) we have
\[
\langle \pi_\lambda\otimes \omega_n, I^{\UU_{n}}(\tau\otimes\pi')\rangle_{\UU_{n}(\Fq)}\le m(\pi_{\lambda_*}, \Theta_{m, m-\lambda_1+1}(\pi'))=0.
\]
 Assume that $\pi'=\pi_{\mu}$ is unipotent, where $\mu$  is a partition of $m$. Let $\mu_1$ be the largest part of $\mu$, and take an arbitrary $\mu_0\ge \max\{\lambda_1,\mu_1\}$. Put  $\mu^*=[\mu_0,\mu]$. Then $\Theta_{m+\mu_0,m}(\pi_{\mu^*})=\pi_\mu$. Consider the see-saw diagram
\[
\setlength{\unitlength}{0.8cm}
\begin{picture}(20,5)
\thicklines
\put(6.8,4){$\UU_n\times \UU_n$}
\put(7.5,1){$\UU_n$}
\put(12.7,4){$\UU_{n+\mu_0+1}$}
\put(12,1){$\UU_{n+\mu_0}\times \UU_1$}
\put(7.7,1.5){\line(0,1){2.1}}
\put(12.8,1.5){\line(0,1){2.1}}
\put(8,1.5){\line(2,1){4.2}}
\put(8,3.7){\line(2,-1){4.2}}
\end{picture}
\]
Note that by Lemma \ref{irr}, $I^{\UU_{n}}(\tau\otimes\pi_\mu)$ is irreducible.
Similar to the proof of (i), one has
 \begin{align*}
 & \langle \pi_\lambda\otimes \omega_n, I^{\UU_{n}}(\tau\otimes\pi_\mu)\rangle_{\UU_{n}(\Fq)}\\
= & \langle \pi_\lambda,  I^{\UU_{n}}(\tau\otimes\pi_\mu)\otimes\omega_n\rangle_{\UU_{n}(\Fq)}\\
 =&  \langle \pi_\lambda, \Theta_{n+\mu_0,n}(I^{\UU_{n+\mu_0}}(\tau\otimes\pi_{\mu^*}))\otimes\omega_n \rangle _{\UU_{n}(\Fq)}\\
=& \langle  \Theta_{n, n+\mu_0+1}(\pi_{\lambda}), I^{\UU_{n+\mu_0}}(\tau\otimes\pi_{\mu^*}))\rangle _{\UU_{n+\mu_0}(\Fq)}\\
=& m( \Theta_{n, n+\mu_0+1}(\pi_\lambda), \pi_{\mu^*}).
\end{align*}
By Proposition \ref{7.1},
 \[
 \Theta_{n, n+\mu_0+1}(\pi_\lambda)=\bigoplus_{\mbox{\tiny$\begin{array}{c} {}^t\lambda \textrm{ and }{}^t\tilde{\lambda}\textrm{ are 2-transverse}\\
|\tilde{\lambda}|=n+\mu_0+1\end{array}$}}\pi_{\tilde{\lambda}}.
 \]
 Thus
 \[
 m(\Theta_{n, n+\mu_0+1}(\pi_\lambda), \pi_{\mu^*})=\sum_{\mbox{\tiny$\begin{array}{c} {}^t\lambda \textrm{ and }{}^t\tilde{\lambda}\textrm{ are 2-transverse}\\
|\tilde{\lambda}|=n+\mu_0+1\end{array}$}}m(\pi_{\tilde{\lambda}}, \pi_{\mu^*}).
 \]
 If ${}^t\lambda$ and ${}^t\tilde{\lambda}$ are 2-transverse, then $\tilde{\lambda}$ has at most $k+1$ parts.
Since $m+\mu_0=(n+\mu_0+1)-(k+1)$, applying Theorem \ref{th1} gives that
\begin{equation}\label{eq7.1}
m(\pi_{\tilde{\lambda}}, \pi_{\mu^*})=\left\{\begin{array}{ll} 1, & \textrm{if }\mu^*=\tilde{\lambda}',\\
0, & \textrm{otherwise,}
\end{array}
\right.
\end{equation}
where $\tilde{\lambda}'$ is the partition of $n+\mu_0+1$ obtained by removing the first column of $\tilde{\lambda}$. Note that if
$m(\pi_{\tilde{\lambda}}, \pi_{\mu^*})\neq 0$ then necessarily $\tilde{\lambda}$ has $k+1$ parts. Clearly
\[
\lambda^*:={}^t[k+1,{}^t\mu^*]
\]
is the unique partition  $\tilde{\lambda}$ of $n+\mu_0+1$ satisfying that $\tilde{\lambda}'=\mu^*$, if such partition exists at all. Namely, $\lambda^*$ is obtained by adding a column of $k+1$ squares to the left of $\mu^*$. Thus  we may rephrase (\ref{eq7.1}) as
\[
m(\pi_{\tilde{\lambda}}, \pi_{\mu^*})=\left\{\begin{array}{ll} 1, & \textrm{if }\tilde{\lambda}=\lambda^*,\\
0, & \textrm{otherwise.}
\end{array}
\right.
\]
Noting that the first row of $\lambda^*$ has $\mu_0+1$ squares, i.e. $\lambda^*$ has $\mu_0+1$ columns, it is not hard to see that if ${}^t\lambda$ and ${}^t\lambda^*$ are 2-transverse, then necessarily
\[
\lambda^*=[\mu_0+1, \lambda],
\]
which together with $(\lambda^*)'=\mu^*$ imply that $\lambda'=\mu$. In summary,
we conclude that
\[
 m(\pi_\lambda, \pi_\mu)=m( \Theta_{n, n+\mu_0+1}(\pi_\lambda), \pi_{\mu^*})=\left\{\begin{array}{ll} 1, & \textrm{if }\mu=\lambda',\\
0, & \textrm{otherwise,}
\end{array}
\right.
\]
which finishes the proof of (ii).
\end{proof}

\end{document}